\documentclass[a4paper,twoside,DIV=12,headings=small,%
headsepline=true,toc=bib
]{scrartcl}
\pdfoutput=1
\ifpdfoutput{\usepackage{cmap}}{}
\usepackage[utf8]{inputenc}
\usepackage[T1,T2A]{fontenc}
\usepackage[inner=2cm,outer=2cm,top=2cm,bottom=2cm,footskip=.5cm,headsep=.3cm]{geometry}

\def\thispapertitle {Hurwitz-type matrices of doubly infinite series}

\usepackage{mathtools}
\usepackage{paralist}
\usepackage{colonequals}

\usepackage[usenames,svgnames]{xcolor}
\usepackage{float}

\usepackage{extarrows}
\usepackage{mathdots}
\usepackage[selected=true,shrink=15,stretch=60,verbose=true,tracking=alltext,letterspace=0
]{microtype}
\SetTracking[
    spacing = {170*,-25*,}
]{ font = {*/PTSerif-TLF/*/it/*} }{ 10 }

\usepackage{setspace}
\ifpdfoutput{
\definecolor{mybrown}{HTML}{D02000}
\usepackage[colorlinks=false,linkcolor=mybrown, pdfborderstyle={/S/D/D[2 1]/W .5},
citecolor=DarkGreen,linktocpage=true,unicode=true,pdftex,
        pdfversion=1.5,
	pdftitle={\thispapertitle},
	pdfauthor={Alexander Dyachenko},
	pdfsubject={MSC(2010) 30C15, 30B10, 40A05},
	pdfkeywords={Total positivity, Polya frequency sequence, Total nonnegativity,
            Hurwitz matrix, Generalized Hurwitz matrix, Doubly infinite series},
	pdfpagemode=UseOutlines,pdflang=en-GB,
        ]{hyperref}
        \hypersetup{
            pdftitle={\thispapertitle}
        }
}
{\usepackage[colorlinks=true,citecolor=DarkGreen,linktoc=page,unicode=true
]{hyperref}
}
\usepackage[russian,british]{babel}
\usepackage{amsmath,amsthm}
\usepackage{amssymb,amsfonts}
\usepackage
{fourier}
\usepackage{bm}

\RequirePackage[scaled=0.94]{PTSansCaption}  
\RequirePackage[scaled=0.94]{PTSerifCaption} 
\RequirePackage[scaled=0.94]{PTSans}
\RequirePackage[scaled=0.94]{PTSerif}

\newcommand{\navy}{}
\newcommand{\gree}{}

\newcommand{\sign}{\operatorname{sign}}

\renewcommand{\le}{\leqslant}
\renewcommand{\ge}{\geqslant}

\newcommand{\ww}{\quad\text{where}\quad}
\newcommand{\an}{\quad\text{and}\quad}

\renewcommand{\Im}{\operatorname{Im}}

\newtheorem{theorem}{\usekomafont{subparagraph}Theorem}

\newtheorem{lemma}[theorem]{\usekomafont{subparagraph}Lemma}

\newtheorem{corollary}[theorem]{\usekomafont{subparagraph}Corollary}

\theoremstyle{definition}
\newtheorem{fact}{\usekomafont{subparagraph}Fact}

\newtheorem{remark}[theorem]{\usekomafont{subparagraph}Remark}
\newtheorem*{remark*}{\usekomafont{subparagraph}Remark}

\newtheorem*{definition}{\usekomafont{subparagraph}Definition}

\numberwithin{paragraph}{section}

\makeatletter
\def\blfootnote{\gdef\@thefnmark{}\@footnotetext}
\makeatother

\title{\thispapertitle%
    \thanks{This work was supported by the Einstein Foundation Berlin and the European Research
        Council under the European Union's Seventh Framework Programme (FP7/2007--2013)/ERC
        grant agreement no.~259173.} }

\author{\gree\fontfamily{PTSansCaption-TLF}\selectfont\normalsize\vspace{1em}Alexander
    Dyachenko%
    \footnote{\emph{TU-Berlin, Institut f\"ur Mathematik, Sekr.~MA 4-2, Straße des 17. Juni 136,
        10623 Berlin, Germany.\newline        
        \href{mailto:diachenko@sfedu.ru}{diachenko@sfedu.ru},
        \href{mailto:dyachenk@math.tu-berlin.de}{dyachenk@math.tu-berlin.de} }}
}
\date{\gree\fontfamily{PTSansCaption-TLF}\selectfont\vspace{-.4em}{\small\today\vskip -2.5em}}

\usepackage{scrpage2}
\clearscrheadfoot
\ohead{\pagemark}
\cehead{\quad A.\,Dyachenko --- \thispapertitle}
\cohead{\leftmark}
\setheadsepline{.4pt}
\deftripstyle{my_titlepagestyle}{}{}{}{}{\pagemark}{}

\addtokomafont{title}{\navy\fontfamily{PTSans-TLF}\selectfont}
\addtokomafont{section}{\navy\fontfamily{PTSansCaption-TLF}\selectfont\mdseries\large}
\addtokomafont{subsection}{\navy\fontfamily{PTSansCaption-TLF}\selectfont\mdseries}
\addtokomafont{paragraph}{\navy\fontfamily{PTSansCaption-TLF}\selectfont\mdseries}
\addtokomafont{subparagraph}{\navy\fontfamily{PTSans-TLF}\selectfont}
\addtokomafont{pageheadfoot}{\gree\fontfamily{PTSansCaption-TLF}\selectfont}
\addtokomafont{pagenumber}{\gree\fontfamily{PTSansCaption-TLF}\selectfont}
\addtokomafont{headsepline}{\gree}
\addtokomafont{footsepline}{\gree}

\setkomafont{sectionentry}{\vspace{-.7em}\fontsize{8pt}{1em}\fontfamily{PTSansCaption-TLF}\selectfont}

\makeatletter
\renewcommand{\tocbasic@@before@hook}{\vspace{-1.5em}\parskip0pt\itemsep0pt\setstretch{1}}
\makeatother

\newcommand{\Cp}{\ensuremath{\mathbb{C}_+}}

\begin{document}
\automark{section}
\maketitle

\setstretch{1.2}

\begin{abstract}
    This paper show that two doubly infinite series generate a totally nonnegative Hurwitz-type
    matrix if and only if their ratio represents an~$\mathcal S$-functions of a certain kind.
    The doubly infinite case needs a specific approach, since the ratios have no correspondent
    Stieltjes continued fraction. Another forthcoming publication (see
    \href{http://arxiv.org/abs/1608.04440}{Dyachenko, arXiv:1608.04440}) offers a shorter
    improved version of this result as well as its application to the Hurwitz stability.
    Nevertheless, the proof presented here illustrates features of totally nonnegative
    Hurwitz-type matrices better.
\end{abstract}
\vspace{5pt}
        \\[2pt]
\section{Introduction}
This paper offers a weaker version of the main result of the publication~\cite{Dyachenko16b}.
One of its features (and its key difference from~\cite{Dyachenko16b}) is that it tries to make
total nonnegativity the cornerstone. Accordingly, some well-known properties are rederived
directly from estimates of matrix minors, so the proofs turn to be more self-contained.
\begin{definition}
    A doubly (\emph{i.e.} two-way) infinite sequence $\big(f_n\big)_{n=-\infty}^\infty$ is
    called \emph{totally positive} if all minors of the (four-way infinite) Toeplitz matrix
    \begin{equation*}
        \begin{pmatrix}
            \ddots & \vdots & \vdots & \vdots & \vdots & \vdots & \iddots\\
            \hdots & f_0 & f_1 & f_2 & f_3 & f_4 & \hdots\\
            \hdots & f_{-1} & f_0 & f_1 & f_2 & f_3 & \hdots\\
            \hdots & f_{-2} & f_{-1} & f_0 & f_1 & f_2 & \hdots\\
            \hdots & f_{-3} & f_{-2} & f_{-1} & f_0 & f_1 & \hdots\\
            \hdots & f_{-4} & f_{-3} & f_{-2} & f_{-1} & f_0 & \hdots\\
            \iddots & \vdots & \vdots & \vdots & \vdots & \vdots & \ddots
        \end{pmatrix}
        \equalscolon T(f),
        \ww
        f(z)\colonequals\sum_{\mathclap{n=-\infty}}^\infty \ f_nz^n
    \end{equation*}
    are nonnegative (\emph{i.e.} the matrix is \emph{totally nonnegative}).
\end{definition}
\begin{theorem}[{Edrei~\cite{Edrei}%
        \footnote{Another proof is given in~{\cite[Section~8]{Karlin}}; an earlier
            publication~\cite{AESW} studies the singly infinite case.}}]
    \label{th:E-AESW}
    Let a non-trivial sequence~$\big(f_n\big)_{n=-\infty}^\infty$ be totally positive.
    Then, unless~$f_n=f_0^{1-n}f_1^{n}$ for every~$n\in\mathbb{Z}$, the series~$f(z)$ converges
    in some annulus to a function with the following 
    representation
    \begin{equation}\label{eq:funct_gen_dtps}
        Cz^je^{Az+\frac{A_0}z}\cdot
        \frac{\prod_{\mu>0} \left(1+\frac{z}{\beta_\mu}\right)}
        {\prod_{\nu>0} \left(1-\frac{z}{\delta_\nu}\right)}\cdot
        \frac{\prod_{\mu<0} \left(1+\frac{z^{-1}}{\beta_\mu}\right)}
        {\prod_{\nu>0} \left(1-\frac{z^{-1}}{\delta_\nu}\right)},
    \end{equation}
    where the products converge absolutely, \(j\) is integer and the coefficients satisfy
    \(A,A_0\ge 0\), ~\(C,\beta_\mu,\delta_\nu>0\) for all~\(\mu,\nu\). The converse is also
    true: every function of this form generates (\emph{i.e.} its Laurent coefficients give) a
    doubly infinite totally positive sequence.
\end{theorem}
In the products and sums with inequalities in limits, we assume that the indexing variable
changes in~$\mathbb{Z}$ or in some finite or infinite subinterval of~$\mathbb{Z}$, and that it
additionally satisfies the indicated inequalities. Accordingly, a product or sum can be empty,
finite or infinite. Note that the indexation of four-way infinite matrices affects the
multiplication. Here we adopt the following convention: the uppermost row and the leftmost
column, which appear in representations of such matrices, have the index~$1$ unless another is
stated explicitly.

The so-called Hurwitz-type matrices have applications to stability theory. They are built from
two Toeplitz matrices; more specifically,
\begin{definition}
    The \emph{Hurwitz-type matrix} is a matrix of the form
    \begin{equation}\label{eq:Hpq_def}
        H(p,q)={\begin{pmatrix}
                \ddots & \vdots & \vdots & \vdots & \vdots & \vdots & \vdots & \iddots\\
                \hdots & a_0& a_1& a_2 & a_3 & a_4 & a_5 & \hdots\\
                \hdots & b_0& b_1& b_2 & b_3 & b_4 & b_5 & \hdots\\
                \hdots & a_{-1} & a_0 & a_1 & a_2 & a_3 & a_4 & \hdots\\
                \hdots & b_{-1} & b_0 & b_1 & b_2 & b_3 & b_4 & \hdots\\
                \hdots & a_{-2} & a_{-1} & a_0 & a_1 & a_2 & a_3 & \hdots\\
                \iddots & \vdots & \vdots & \vdots & \vdots & \vdots & \vdots & \ddots
            \end{pmatrix}},
    \end{equation}
    where~$p(z)=\sum_{k=-\infty}^\infty\,a_kz^k$ and~$q(z)=\sum_{k=-\infty}^\infty\,b_kz^k$ are
    formal power series.
\end{definition}
Recent publications~\cite{HoltzTyaglov,Dyachenko14} have shown that a criterion relevant to
Theorem~\ref{th:E-AESW} holds for the Hurwitz-type matrices. The main goal of the present study
is to give an extension of that criterion: to determine conditions on the power series~$p(z)$
and~$q(z)$ necessary and sufficient for total nonnegativity of the matrix~$H(p,q)$. Like in the
earlier studied cases, one of the conditions is that the ratio~$\frac{q(z)}{p(z)}$ maps the
upper half-plane~$\Cp\coloneqq\{z\in\mathbb C:\Im z>0\}$ into itself. To give a more precise
statement, let us introduce the following class of functions:
\begin{definition}
    A function~$F(z)$ is called an $\mathcal{S}$-function if it is holomorphic and
    satisfies~$\Im z\cdot\Im F(z)\ge 0$ for all~$z\not\le0$ and if additionally~$F(z)\ge0$
    wherever~$z>0$.
\end{definition}
The straightforward corollary of the definition is that~$F(\overline z)=\overline{F(z)}$ for
each~$\mathcal S$-function~$F(z)$ wherever it is regular. We need a subclass of
$\mathcal{S}$-functions introduced in the following lemma.
\begin{lemma}\label{lemma:prop_S1}
    Let~$p(z)$ and~$q(z)$ be two functions of the form~\eqref{eq:funct_gen_dtps}; then their
    ratio~$F(z)=\frac{q(z)}{p(z)}$ is an $\mathcal S$-function if and only if there exists a
    function~$g(z)$ of the form~\eqref{eq:funct_gen_dtps}, such that
    \[
    \begin{gathered}
        \frac{p(z)}{g(z)}= a_0{\prod_{\nu>0} \left(1+\dfrac{z}{\alpha_\nu}\right)}
        {\prod_{\nu<0}\left(1+\dfrac{z^{-1}}{\alpha_\nu}\right)}, \quad \frac{q(z)}{g(z)}=
        b_0{\prod_{\mu>0} \left(1+\dfrac{z}{\beta_\mu}\right)}
        {\prod_{\mu<0}\left(1+\dfrac{z^{-1}}{\beta_\mu}\right)} \an
        \\
        0<\cdots<\alpha_{-2}^{-1}<\beta_{-1}^{-1}<\alpha_{-1}^{-1}
        <\beta_{1}<\alpha_{1}<\beta_{2}<\alpha_2<\cdots;
    \end{gathered}
    \]
    if the sequence of~$\mu$ terminates on the left at~$\mu_0$, then~$\beta_{\mu_0}$ can be
    positive or zero%
    \footnote{When~$\beta_{\mu_0}=0$, the corresponding
        factor~$\Big(1+\frac{z}{\beta_{\mu_0}}\Big)$ needs to be replaced by the factor~$z$.}
    and the sequence of~$\nu$ also terminates on the left at~$\mu_0$.
\end{lemma}
This lemma is an analogue of a theorem due to Kre\u{\i}n, see~\cite[p.~308]{Levin}. In other
words, under its conditions the function~$F(z)$ can be expressed as presented below
in~\eqref{eq:order_alphas_betas} or~\eqref{eq:order_alphas_betas_mer}. The chain inequality
means that zeros of~$\frac{p(z)}{g(z)}$ and~$\frac{q(z)}{g(z)}$ are \emph{interlacing}, that is
all zeros of each of the functions are real and separated by zeros of another.
Lemma~\ref{lemma:prop_S1} provides another reformulation of the item~\eqref{item:m1} in our main
result:
\begin{theorem}\label{th:main1}
    If~$a_0\ne 0$, then the following conditions are equivalent:
    \begin{enumerate}[\upshape(a)]
    \item \label{item:m1}%
        The series~$p(z)=\sum_{k=-\infty}^\infty\,a_kz^k$
        and~$q(z)=\sum_{k=-\infty}^\infty\,b_kz^k$ converge in some common annulus to functions
        of the form~\eqref{eq:funct_gen_dtps} and their ratio~$F(z)=\frac{q(z)}{p(z)}$ is an
        $\mathcal S$-function.
    \item \label{item:m3}%
        The matrix~$H(p,q)$ is totally nonnegative and there exists $k\in\mathbb{Z}$ such
        that~$a_k^2\ne a_{k-1}a_{k+1}$.
    \end{enumerate}
\end{theorem}
\begin{remark}\label{rem:degenerating_series}
    The conditions~$a_0\ne 0$ and~$a_k^2= a_{k-1}a_{k+1}$ for all~$k$ are equivalent to the
    condition~$a_k= a_0^{1-k}a_1^{k}\ne 0$ for all~$k$. This case is excluded in~\eqref{item:m3}
    of Theorem~\ref{th:main1} as corresponding to the divergence of the power series~$p(z)$ and,
    unless~$q(z)\equiv 0$, of the series~$q(z)$ by Lemma~\ref{lemma:ann_convergence} (in
    fact, we have~$a_0q(z)\equiv b_0p(z)$).
\end{remark}

Both papers~\cite{HoltzTyaglov,Dyachenko14} exploit a relation to the matching moment problem
trough the Hurwitz transform (see e.g.~\cite[p.~44]{ChebMei} or~\cite[p.~427]{HoltzTyaglov}). In
turn, doubly infinite series do not allow conducting the same procedure due to the lack of the
matching moment problem. To get around the difficulty, we apply a modification of the
technique~\cite{AESW,Karlin} developed in Section~\ref{sec:poles-of-p-and-q-when-H-is-TNN} for
factoring out a totally nonnegative Toeplitz matrix from a totally nonnegative Hurwitz-type
matrix. Another key point is Lemma~\ref{lemma:H_TNN_F_has_form}, which characterizes the
function corresponding to the resulting Hurwitz-type matrix by extending a fact known for
polynomials, see e.g.~\cite[Lemma~3.4]{Wa}. The latter step is based on simple but effective
Lemma~\ref{lemma:T_via_HH}; the converse to this lemma is proved in~\cite{Dyachenko16b}.

By writing that a function has one of the above representations, we assume that the involved
products are locally uniformly convergent unless the converse is stated explicitly. In the above
theorems, the convergence follows from the total nonnegativity of the involved matrices. The
condition of convergence is well-known and can be expressed as the following theorem.
\begin{theorem}[see \emph{e.g.}~{\cite[pp.~7--13,~21]{Levin}}]\label{th:canonical_product}
    The infinite product~$\prod_{\nu=0}^\infty \left(1+\frac{z}{\alpha_\nu}\right)$,
    converges uniformly in~$z$ varying in compact subsets of~$\mathbb{C}$ if and only if the
    series~$\sum_{\nu=0}^\infty \frac 1{|\alpha_\nu|}$ converges. If so, then for
    any~$\varepsilon>0$ the estimates
    \( \prod_{\nu=0}^\infty \left|1+\frac{z}{\alpha_\nu}\right| < C e^{\varepsilon R} \)
    and, outside exceptional disks with an arbitrarily small sum of radii,
    \( \prod_{\nu=0}^\infty \left|1+\frac{z}{\alpha_\nu}\right| > C e^{-\varepsilon R} \)
    provided that~$|z|\le R$ and the positive numbers~$R$ and~$C$ are big enough.
\end{theorem}

\section{Basic properties of infinite Toeplitz and Hurwitz-type matrices}

An important property of totally positive sequences is that they have no gaps, i.e. no zero
coefficients between non-zero coefficients:
\begin{lemma}
    \label{lemma:ann_tnn_zeros}
    If~$p(z)=\sum_{k=-\infty}^\infty a_kz^k$ and the matrix~$T(p)$ is totally nonnegative,
    then~$a_k=0\ne a_{k+1}\implies a_{k-n}=0$
    and~$a_k=0\ne a_{k-1}\implies a_{k+n}=0$ for all~$n\in\mathbb{Z}_{>0}$.
\end{lemma}
In other words, if any of the coefficients of~$p(z)$ turns to zero, then all coefficients to the
left or to the right must be zero. In this case,~$p(z)$ is either a Laurent polynomial or a
singly infinite series with no gaps. Lemma~\ref{lemma:ann_tnn_zero_minor} exploits an analogous
property of minors of~$T(p)$.
\begin{proof}[Proof of Lemma~\ref{lemma:ann_tnn_zeros}]
    Note that all coefficients of~$p(z)$ are nonnegative. If for some~$k\in\mathbb{Z}$ the
    coefficient~$a_k$ is zero and the neighbouring one~$a_{k-1}$ is nonzero, then for
    each integer~$n\ge k$ the inequality
    \[
    0\le
    \begin{vmatrix}
        a_k & a_{n+1}\\
        a_{k-1} & a_n
    \end{vmatrix}
    = - a_{n+1} a_{k-1}\le 0
    \]
    yields that $a_{n+1}=0$. Conversely, if~$a_{n}=0\ne a_{n+1}$ for some~$n$, then from the
    same inequality we have~$a_{k-1}=0$ for all~$k\le n$.
\end{proof}
\begin{lemma}
    \label{lemma:ann_tnn_zero_minor}
    If a series~$p(z)=\sum_{k=-\infty}^\infty a_kz^k$ is such that the matrix~$T(p)$ is totally
    nonnegative, and~$0\ne a_{k-1}^2=a_{k}a_{k-2}$ for some integer~$k$, then~$p(z)$ does not
    converge to any function holomorphic in~$\mathbb{C}\setminus\{0\}$.
\end{lemma}
\begin{proof}
    The condition~$0\ne a_{k-1}^2=a_{k}a_{k-2}$ implies that~$a_{k}$, $a_{k-1}$ and~$a_{k-2}$
    are nonzero. Unless~$a_j=a_0^{1-j}a_1^j$ for all~$j$, there exists an integer~$i$ such
    that~$a_{i-1}a_{k}\ne a_{k-1}a_{i}$ (see Remark~\ref{rem:degenerating_series}). Therefore,
    assuming~$i>j\coloneqq k$ gives us the following relation:
    \begin{equation} \label{eq:Toeplitz_minor}
        \begin{aligned}
            0\le
            \begin{vmatrix}
                a_{i-2}& a_{i-1} & a_{i}\\
                a_{j-2}& a_{j-1} & a_{j}\\
                a_{j-3}& a_{j-2} & a_{j-1}
            \end{vmatrix}
            &= \frac{1}{a_{j-1}}
            \begin{vmatrix}
                a_{i-2}    & a_{i-1} & a_{i}\\
                a_{j-2}    & a_{j-1} & a_{j}\\
                a_{j-1}a_{j-3}-a_{j-2}^2& 0 & 0
            \end{vmatrix}
            \\
            &= -a_{j-1}^{-1} \left(a_{i-1}a_j-a_{i}a_{j-1}\right)
            \left(a_{j-2}^2-a_{j-1}a_{j-3}\right)\le 0.
        \end{aligned}
    \end{equation}
    Consequently,~$a_{j-2}^2=a_{j-1}a_{j-3}\ne 0$.
    Furthermore,~$a_{i-1}a_{j-1}-a_{i}a_{j-2}=\frac{a_{j-1}}{a_{j}}\left(a_{i-1}a_{j}-a_{i}a_{j-1}\right)\ne0$,
    and hence~\eqref{eq:Toeplitz_minor} is valid for~$j=k-1$. Sequentially
    letting~$j=k-1,k-2,\dots$ in~\eqref{eq:Toeplitz_minor} thus implies
    that~$a_{j-1}^2=a_{j}a_{j-2}\ne0$ for each~$j<k$. According to the ratio test, the
    series~$p(z)$ diverges in the disk of the
    radius~$\lim_{j\to-\infty}\frac{a_{j-1}}{a_{j}}=\frac{a_{k-1}}{a_{k}}>0$. The proof for the
    case~$i+1<j\coloneqq k$ is analogous: since
    \begin{equation*}
    0\le
    \begin{vmatrix}
        a_{j-1}& a_{j} & a_{j+1}\\
        a_{j-2}& a_{j-1} & a_{j}\\
        a_{i-1}& a_{i} & a_{i+1}
    \end{vmatrix}
    =
    \frac{1}{a_{j}}
    \begin{vmatrix}
        a_{j-1}& a_{j} & a_{j+1}\\
        0     & 0     & a_{j}^2-a_{j+1}a_{j-1}\\
        a_{i-1}& a_{i} & a_{i+1}
    \end{vmatrix}
    =
    -a_{j}^{-1}
    \left(a_{i}a_{j-1}-a_{i-1}a_j\right)
    \left(a_{j}^2-a_{j+1}a_{j-1}\right)\le 0,
    \end{equation*}
    the equality~$a_{j}^2=a_{j+1}a_{j-1}\ne 0$ holds. Due
    to~$a_{i-1}a_{j+1}-a_{i}a_{j}=\frac{a_{j+1}}{a_j}\left(a_{i-1}a_{j}-a_{i}a_{j-1}\right)\ne0$,
    the same result is true for each~$j=k+1,k+2,\dots$. Consequently, the series~$p(z)$ has the
    radius of convergence equal
    to~$\lim_{j\to+\infty}\frac{a_{j-1}}{a_{j}}=\frac{a_{k-1}}{a_{k}}>0$.
\end{proof}
In fact, applying a version of the Sylvester determinant identity (see
\emph{e.g.}~\cite[formula~(5.1), p.~136]{Pinkus}) reduces the proof of
Lemma~\ref{lemma:ann_tnn_zero_minor} to the proof of Lemma~\ref{lemma:ann_tnn_zeros}. The
following lemma is an analogue of Lemma~\ref{lemma:ann_tnn_zeros} for Hurwitz-type matrices.
\begin{lemma}\label{lemma:ann_convergence_term}
    Let~$H(p,q)$ be totally nonnegative. If both series are non-trivial and one of them has a
    zero coefficient, then both series terminate on the same side. More specifically,
    \begin{compactitem}
    \item $a_k=0\ne a_{k-1}\implies b_{k+r}=0$,
    \item $a_k=0\ne a_{k+1}\implies b_{k-r+1}=0$,
    \item $b_k=0\ne b_{k-1}\implies a_{k+r-1}=0$,
    \item $b_k=0\ne b_{k+1}\implies a_{k-r}=0$
    \end{compactitem}
    for all~$r=1,2,\dots$. In addition, $a_{k-1},a_k\ne 0\implies b_k\ne 0$
    and~$b_k,b_{k+1}\ne 0\implies a_{k}\ne 0$.
\end{lemma}
\begin{proof}
    By Lemma~\ref{lemma:ann_tnn_zeros}, the condition~$a_k=0\ne a_{k-1}$ for some~$k$
    yields~$a_{k+r-1}=0$; therefore,
    \[
    0\le\begin{vmatrix}
        b_{k}   & b_{k+r} \\
        a_{k-1} & a_{k+r-1}
    \end{vmatrix}
    =-a_{k-1}b_{k+r}\le 0 \implies b_{k+r}=0.
    \]
    The next three implications follow analogously from evaluating (respectively) the minors
    \[
    \begin{vmatrix}
        a_{k-r+1} & a_{k+1} \\
        b_{k-r+1} & b_{k+1}
    \end{vmatrix}
    ,
    \quad
    \begin{vmatrix}
        a_{k-1} & a_{k+r-1} \\
        b_{k-1} & b_{k+r-1}
    \end{vmatrix}
    \an
    \begin{vmatrix}
        b_{k-r+1} & b_{k+1} \\
        a_{k-r} & a_{k}
    \end{vmatrix}
    .
    \]
    The assumption~$a_{k-1},a_k\ne0=b_k$ gives that the minors
    \[
    \begin{vmatrix}
        a_{k-r}& a_{k}\\
        b_{k-r}& b_{k}
    \end{vmatrix}
    =-a_kb_{k-r}
    \an
    \begin{vmatrix}
        b_{k}& b_{k+r}\\
        a_{k-1}& a_{k+r-1}
    \end{vmatrix}
    =-a_{k-1}b_{k+r}
    \]
    vanish for all~$r=1,2,\dots$, which is inconsistent with the non-triviality of~$q(z)$.
    Similarly,~$b_k,b_{k+1}\ne0=a_k$ gives vanishing of
    \[
    \begin{vmatrix}
        b_{k-r+1}& b_{k+1}\\
        a_{k-r}& a_{k}
    \end{vmatrix}
    =-a_{k-r}b_{k+1}
    \an
    \begin{vmatrix}
        a_{k}& a_{k+r}\\
        b_{k}& b_{k+r}
    \end{vmatrix}
    =-b_ka_{k+r}
    \]
for all~$r=1,2,\dots$, which contradicts to the non-triviality of~$p(z)$.
\end{proof}
The next fact is relevant to Lemma~\ref{lemma:ann_tnn_zero_minor} for series terminating on the
right, and its proof can also be conducted with the help of the Sylvester determinant identity.
Note that Lemma~\ref{lemma:H_tnn_zero_minor} admits a reformulation for series terminating on
the left.
\begin{lemma}\label{lemma:H_tnn_zero_minor}
    Given series~$p(z)=\sum_{k=-\infty}^{n}a_kz^k$ such that~$a_n\ne 0$
    and~$q(z)=\sum_{k=-\infty}^{\infty}b_kz^k\not\equiv 0$ suppose that~$H(p,q)$ is totally
    nonnegative and~$a_{k-1}b_{k}=a_{k}b_{k-1}\ne 0$ for some~$k<n$. Then
    \begin{compactenum}[\upshape(a)]
    \item\label{item:Hurwitz_minor_b}%
        the inequality~$a_{n-1}b_{n}\ne a_{n}b_{n-1}$ implies that neither~$p(z)$ nor~$q(z)$ can
        converge to any function holomorphic in~$\mathbb{C}\setminus\{0\}$;
    \item\label{item:Hurwitz_minor_c}%
        if at least one of the series~$p(z),q(z)$ converges to a function holomorphic
        in~$\mathbb{C}\setminus\{0\}$, then~$p(z)=\frac{a_{n}}{b_{n}}q(z)$.
    \end{compactenum}
\end{lemma}
\begin{proof}
    By Lemma~\ref{lemma:ann_tnn_zeros}, if~$a_{k-1},a_{n}\ne 0$, then all the
    numbers~$a_{k-1},a_{k},\dots,a_{n}$ are strictly positive.
    Since the last term~$- a_{j}b_{k-1}$ in the right-hand side of
    \[
    0\le
    \begin{vmatrix}
        a_{k-1}& a_{j}\\
        b_{k-1}& b_{j}
    \end{vmatrix}
    =a_{k-1}b_{j} - a_{j}b_{k-1},
    \]
    is nonzero for each~$j=k,k+1,\dots,n$, all the numbers~$b_{k-1},b_{k},\dots,b_{n}$ are
    strictly positive as well. Furthermore, Lemma~\ref{lemma:ann_convergence_term} yields
    $b_{n+2}=b_{n+3}=\dots=0$. The condition~$a_{k-1}b_{k}=a_kb_{k-1}\ne 0$ implies that
    \begin{equation}\label{eq:minor_Hurwitz1}
    0\le
    \begin{vmatrix}
        a_{k-1}& a_{k} & a_{n}\\
        b_{k-1}& b_{k} & b_{n}\\
        a_{k-2}& a_{k-1} & a_{n-1}
    \end{vmatrix}
    =
    \begin{vmatrix}
        a_{k-1}& a_{k} & a_{n}\\
        0& 0 & b_{n}-\frac{b_{k-1}}{a_{k-1}}a_n\\
        a_{k-2}& a_{k-1} & a_{n-1}
    \end{vmatrix}
    =
    \left(\frac{b_{k-1}}{a_{k-1}}a_n-b_{n}\right)
    \begin{vmatrix}
        a_{k-1}& a_{k}\\
        a_{k-2}& a_{k-1}
    \end{vmatrix}
    \le 0
    \end{equation}
    and, due to~$a_{n+1}=0$, for each~$i\le k-2$
    \begin{equation}\label{eq:minor_Hurwitz2}
    0\le
    \begin{vmatrix}
        a_{i+n+1-k}& a_{n} & a_{n+1}\\
        a_{i}& a_{k-1} & a_{k}\\
        b_{i}& b_{k-1} & b_{k}
    \end{vmatrix}
    =
    -a_n (a_ib_k-a_{k}b_{i})
    \le 0
    \implies a_ib_k=a_{k}b_{i}
    .
    \end{equation}
    The last equality shows that~$a_i\ne 0 \iff b_i\ne 0$ for all~$i\le k-2$ since~$a_k$
    and~$b_k$ are nonzero. Moreover, \eqref{eq:minor_Hurwitz2} implies that the series~$p(z)$
    and~$q(z)$ converge in the same domain:
    \[
    q(z)=\frac{b_k}{a_k}p(z) + \sum_{i=k}^{n} \left(b_i - \frac{b_k}{a_k} a_i\right) z^i +b_{n+1}z^{n+1}.
    \]
    
    Let us prove~\eqref{item:Hurwitz_minor_b}.
    Multiplying~\( \frac{a_{k-1}}{a_{n-1}}\ge \frac{b_{k-1}}{b_{n-1}} \)
    and~\( a_{n-1}b_n>b_{n-1}a_n\) gives the inequality~\({a_{k-1}}b_{n}>a_n{b_{k-1}}\).
    Substituting it into the relation~\eqref{eq:minor_Hurwitz1} yields~$a_{k-1}^2=a_{k}a_{k-2}$,
    so the series~$p(z)$ does not converge to any function holomorphic
    in~$\mathbb{C}\setminus\{0\}$ by Lemma~\ref{lemma:ann_tnn_zero_minor}. Neither does the
    series~$q(z)$, because it coincides with~$\frac{b_k}{a_k}p(z)$ up to a Laurent polynomial.
    
    To prove~\eqref{item:Hurwitz_minor_c}, note that the convergence of~$p(z)$ for~$z\ne 0$
    implies~$a_{k-1}^2\ne a_{k}a_{k-2}$ by Lemma~\ref{lemma:ann_tnn_zero_minor}. Due
    to~\eqref{eq:minor_Hurwitz1}, the coefficients of~$p(z)$ and~$q(z)$
    satisfy~$a_{k-1}b_{n}=a_{n}b_{k-1}$. We will get a contradiction to this equality
    whenever~$a_{i}b_{n}> a_{n}b_{i}$ for some~$i=k,k+1,\dots n-1$:
    \[
    a_{k-1}b_{n}=\frac{a_{k-1}}{a_{i}}a_{i}b_{n}
    > \frac{a_{k-1}}{a_{i}}a_{n}b_{i}
    \ge a_n b_{k-1}
    \]
    since~$a_{k-1}b_{i}\ge b_{k-1}a_{i}$. Summing up,~$a_{i}b_{n}=a_{n}b_{i}$ for each~$i<n$. The
    coefficient~$b_{n+1}$ is zero, because
    \[
    0\le
    \begin{vmatrix}
        a_{n-1}& a_{n} & a_{n+1}\\
        b_{n-1}& b_{n} & b_{n+1}\\
        a_{n-2}& a_{n-1} & a_{n}
    \end{vmatrix}
    =-b_{n+1}
    \begin{vmatrix}
        a_{n-1}& a_{n}\\
        a_{n-2}& a_{n-1}
    \end{vmatrix}
    \le 0.
    \]
    Consequently,~$q(z)=\frac{b_n}{a_n}p(z)$.
\end{proof}

\begin{theorem}[Whitney, see~\cite{Whitney}]\label{th:whitney}
    Let~$m,n,j$ be positive integers such that~$j\le m$ and
    let~$\displaystyle\big\{a_{ik}\big\}_{1\le i\le m,1\le k\le n}$ be real numbers. The following matrices are
    totally nonnegative simultaneously:
    \[
    \begin{pmatrix}
        1&a_{12}&a_{13}&\hdots& a_{1n}\\
        1&a_{22}&a_{23}&\hdots& a_{2n}\\
        \vdots&\vdots&\vdots&\ddots&\vdots\\
        1&a_{j2}&a_{j3}&\hdots& a_{jn}\\
        0&a_{j+1,2}&a_{j+1,3}&\hdots& a_{j+1,n}\\
        \vdots&\vdots&\vdots&\ddots&\vdots\\
        0&a_{m2}&a_{m3}&\hdots& a_{mn}\\
    \end{pmatrix}
    \an
    \begin{pmatrix}
        a_{12}&a_{13}&\hdots& a_{1n}\\
        a_{22}-a_{12}&a_{23}-a_{13}&\hdots& a_{2n}-a_{1n}\\
        \vdots&\vdots&\ddots&\vdots\\
        a_{j2}-a_{j-1,2}&a_{j3}-a_{j-1,3}&\hdots& a_{jn}-a_{j-1,n}\\
        a_{j+1,2}&a_{j+1,3}&\hdots& a_{j+1,n}\\
        \vdots&\vdots&\ddots&\vdots\\
        a_{m2}&a_{m3}&\hdots& a_{mn}\\
    \end{pmatrix}
    .
    \]
\end{theorem}
\setstretch{1.2}%
\begin{lemma}[{\emph{e.g.}, \cite[p.~6]{Pinkus}}]\label{lemma:TNN_order}
    Given an~$n\times m$ matrix~$M$ let~$\widetilde M$ be the matrix obtained from~$M$ by
    permuting its rows and columns in the opposite order. Then~$M$ and~$\widetilde M$ are
    totally nonnegative simultaneously.
\end{lemma}
\begin{proof}
    Given an arbitrary positive integer~$k\le\min\{n,m\}$ and any two sets of
    integers~$0<i_1<i_2<\dots\penalty10<i_k\le n$ and~$0<j_1<\dots<j_k\le m$ denote the minor of
    the matrix~$M$ with rows~$i_1,i_2,\dots,i_k$ and columns~$j_1,j_2,\penalty10\dots,j_k$ by
    \[
    M\binom{i_1,i_2,\dots,i_k}{j_1,j_2,\dots,j_k}.
    \]
    Then the lemma follows from the identity
    \[
    M\binom{i_1,i_2,\dots,i_k}{j_1,j_2,\dots,j_k}
    =\left((-1)^{\frac{k(k-1)}{2}}\right)^2\cdot
    \widetilde M\binom{n-i_k,n-i_{k-1},\dots,n-i_1}{m-j_k,m-j_{k-1},\dots,m-j_1}.
    \]
\end{proof}

\setstretch{1.15}%
\begin{corollary}
    Let~$p(z)$ and~$q(z)$ be two power series, and
    let~$\check p(z)\coloneqq p\left(\frac 1z\right)$
    and~$\check q(z)\coloneqq q\left(\frac 1z\right)$. Then the mat\-rices~$H(p,q)$
    and~$H(\check q,\check p)$ are totally nonnegative simultaneously.
\end{corollary}
\begin{proof}
    Indeed, assume that~$p(z)=\sum_{n=-\infty}^\infty a_nz^n$\;
    and~$q(z)=\sum_{n=-\infty}^\infty b_nz^n$.\;
    Then~$\check p(z)=\sum_{n=-\infty}^\infty a_{-n}z^n$\;
    and~$\check q(z)=\sum_{n=-\infty}^\infty b_{-n}z^n$, so each submatrix
    of~$H(\check q,\check p)$ coincides with some submatrix of~$H(p,q)$ up to the permutation of
    rows and columns in the opposite order.
\end{proof}

\begin{lemma}\label{lemma:H_zpq_order}
    Let~$p(z)$ and~$q(z)$ be two power series, and let~$\widetilde p(z)= zp(z)$. Then the
    matrices~$H(p,q)$ and~$H(q,\widetilde p)$ coincide up to a shift in indexation.
\end{lemma}
\begin{proof}
    Denote~$p(z)=\sum_{k=-\infty}^\infty\,a_kz^k$ and~$q(z)=\sum_{k=-\infty}^\infty\,b_kz^k$.
    On the one hand,~$\widetilde p(z)=\sum_{n=-\infty}^{\infty} a_{n-1}z^n$ and, hence,
    \[
        H(q,\widetilde p)={\begin{pmatrix}
                \ddots & \vdots & \vdots & \vdots & \vdots & \vdots & \vdots & \iddots\\
                \hdots & b_0& b_1& b_2 & b_3 & b_4 & b_5 & \hdots\\
                \hdots & a_{-1} & a_0& a_1& a_2 & a_3 & a_4 & \hdots\\
                \hdots & b_{-1} & b_0 & b_1 & b_2 & b_3 & b_4 & \hdots\\
                \hdots & a_{-2} & a_{-1} & a_0 & a_1 & a_2 & a_3 & \hdots\\
                \hdots & b_{-2} & b_{-1} & b_0 & b_1 & b_2 & b_3 & \hdots\\
                \hdots & a_{-3} & a_{-2} & a_{-1} & a_0 & a_1 & a_2 & \hdots\\
                \iddots & \vdots & \vdots & \vdots & \vdots & \vdots & \vdots & \ddots
            \end{pmatrix}}.
    \]
    On the other hand, shifting the whole matrix~$H(p,q)$ up results in the same matrix; that
    is,~$H(q,\widetilde p)$ can be obtained by increasing%
    \footnote{There are infinitely many equivalent ways to describe this shift in indexation,
        because the matrix~$H(p,q)$ does not alter when the indices of columns change by~$k$
        and, simultaneously, the indices of rows change by~$2k$ for any integer~$k$.}
    the indices of rows in~$H(p,q)$ by~$1$.
\end{proof}

\section{Poles and exponential factors}
\label{sec:poles-of-p-and-q-when-H-is-TNN}
If one of the series~$p(z)$ or~$q(z)$ is trivial, it converges in the whole plane; in this
special case the total nonnegativity of~$H(p,q)$ does not imply that another series converges in
the same domain. For non-trivial series, the following lemma shows that the total nonnegativity
of~$H(p,q)$ yields the same annulus of convergence for both~$p(z)$ and~$q(z)$.
\begin{lemma}
    \label{lemma:ann_convergence}
    Let power series~$p(z)$ and~$q(z)$ be non-trivial. If the matrix~$H(p,q)$ is totally
    nonnegative, then the series converge in the same annulus, say
    $\{z\in\mathbb{C}:0\le r<|z|<R\le\infty\}$. If~$a_k\ne0$ for all~$k\ll 0$, then
    \[\lim_{k\to-\infty}\frac {b_{k}}{b_{k+1}}= \lim_{k\to-\infty}\frac {a_{k}}{a_{k+1}} =r;\]
    if~$a_k\ne0$ for all~$k\gg 0$, then
    \[\lim_{k\to\infty}\frac {b_{k}}{b_{k+1}}= \lim_{k\to\infty}\frac {a_{k}}{a_{k+1}} =R.\]
    The annulus can be empty, then for all~$k$ the coefficients satisfy~$a_k= a_0^{1-k}a_1^{k}$
    and~$b_k = \frac{b_0}{a_0} a_k =b_0^{1-k}b_1^{k}$, \emph{i.e.} all minors of~$H(p,q)$ of
    order~$\ge2$ vanish.
\end{lemma}
\begin{proof}
    By Lemmata~\ref{lemma:ann_tnn_zeros} and~\ref{lemma:ann_convergence_term}, there are two
    mutually exclusive possibilities: $\exists n\in\mathbb{Z}$ such that~$a_k=b_k=0$ for
    all~$k>n$, or $\exists n\in\mathbb{Z}$ such that~$a_k,b_k>0$ for all~$k>n$. In the former
    case, both series~$p(z)$ and~$q(z)$ converge outside some disk~$|z|\le r<\infty$ and we
    put~$R=\infty$. In the latter case,
    \begin{equation*}
        \begin{vmatrix}
            a_k & a_{k+1}\\
            b_k & b_{k+1}
        \end{vmatrix}
        =a_kb_{k+1} - b_ka_{k+1}
        \ge 0
        \an
        \begin{vmatrix}
            b_k & b_{k+1}\\
            a_{k-1} & a_k
        \end{vmatrix}
        =a_kb_k - a_{k-1}b_{k+1}
        \ge 0
    \end{equation*}
    together give
    \begin{equation}\label{eq:lims_bk_between_ak}
        \frac{a_{k-1}}{a_k} \le \frac{b_k}{b_{k+1}} \le \frac{a_k}{a_{k+1}}
    \end{equation}
    for all~$k>n$, \emph{i.e.} the limits
    \begin{equation}\label{eq:lims_R}
        \lim_{k\to\infty}\frac {a_k}{a_{k+1}} =
        \lim_{k\to\infty}\frac {b_k}{b_{k+1}} \eqqcolon
        R\in\left[\frac {a_{n+1}}{a_{n+2}},+\infty\right]\subset(0,+\infty]
    \end{equation}
    exist. By the ratio test, the radius of convergence of the
    series~$\sum_{k=1}^\infty\,a_kz^k$ and~$\sum_{k=1}^\infty\,b_kz^k$ is equal to~$R$.

    Analogously, there are two mutually exclusive possibilities: $\exists n\in\mathbb{Z}$ such
    that~$a_k=b_k=0$ for all~$k\le n$, or $\exists n\in\mathbb{Z}$ such that~$a_k,b_k>0$ for
    all~$k\le n$. In the former case, both series~$p(z)$ and~$q(z)$ converge in the
    disk~$|z|<R$, that is~$r=0$. In the latter case, the
    inequality~\eqref{eq:lims_bk_between_ak} holds provided that~$k<n$, that is
    \begin{equation}\label{eq:lims_r}
        \lim_{k\to-\infty}\frac {a_k}{a_{k+1}} =
        \lim_{k\to-\infty}\frac {b_k}{b_{k+1}} \eqqcolon
        r\in\left[0,\frac {a_{n-1}}{a_{n}}\right]\subset[0,+\infty).
    \end{equation}
    So, the ratio test implies that~$\sum_{k=-\infty}^0\,a_kz^k$ and~$\sum_{-\infty}^0\,b_kz^k$
    converge absolutely outside the disk~$|z|\le r$. As a result,~$p(z)$ and~$q(z)$ converge
    absolutely if and only if~$r<|z|$ and~$|z|<R$ simultaneously.

    Now, note that any of the equalities~$r=0$ and~$R=\infty$ yields~$r< R$. If~$r>0$
    and~$R<\infty$, then all coefficients of the series~$p(z)$ and~$q(z)$ are positive by
    Lemma~\ref{lemma:ann_tnn_zeros}, and hence the inequality~\eqref{eq:lims_bk_between_ak}
    holds for each integer~$k$. Accordingly,
    \[
    r = \lim_{k\to-\infty}\frac {a_k}{a_{k+1}} \le \lim_{k\to\infty}\frac {a_k}{a_{k+1}} = R.
    \]
    If~$r=R$, then for all~$k\in\mathbb{Z}$ we necessarily have
    \[
    \frac {a_k}{a_{k+1}} = \frac {a_{0}}{a_{1}} = \frac {b_k}{b_{k+1}} = \frac {b_{0}}{b_{1}},
    \quad
    \]
    that is
    \[
    a_k = a_{k-1}\frac{a_1}{a_0} = a_0\left(\frac{a_1}{a_0}\right)^k = a_0^{1-k}a_1^{k}
    \an
    b_k = b_{k-1}\frac{b_1}{b_0} = b_0\left(\frac{a_1}{a_0}\right)^k =\frac{b_0}{a_0} a_k.
    \]
\end{proof}

\begin{lemma}\label{lemma:remove_pole}
    Let the matrix~$H(p,q)$ be totally nonnegative, and let~$r<|z|<R$ be the annulus of
    convergence of the series~$p(z)$ and~$q(z)$ provided by Lemma~\ref{lemma:ann_convergence}.
    Then the inequality~$R<\infty$ implies that the matrix~$H(p_1,q_1)$
    with~$p_1(z)\coloneqq\big(1-\frac{z}{R}\big)p(z)$
    and~$q_1(z)\coloneqq\big(1-\frac{z}{R}\big)q(z)$ is totally nonnegative, and the
    inequality~$r>0$ implies that the matrix~$H(p_2,q_2)$ is totally nonnegative,
    where~$p_2(z)\coloneqq \big(1-\frac{r}{z}\big)p(z)$
    and~$q_2(z)\coloneqq \big(1-\frac{r}{z}\big)q(z)$.
\end{lemma}
\begin{proof}
    Suppose that~$R<\infty$ and consider an arbitrary submatrix of~$H(p,q)$ of the following
    form:
    \[
    \begin{pmatrix}
        a_{m}&a_{m+1}&\hdots&a_{m+2n-1}&a_{m+2n}\\
        a_{0}&a_{1}&\hdots&a_{2n-1}&a_{2n}\\
        b_{0}&b_{1}&\hdots&b_{2n-1}&b_{2n}\\
        a_{-1}&a_{0}&\hdots&a_{2n-2}&a_{2n-1}\\
        b_{-1}&b_{0}&\hdots&b_{2n-2}&b_{2n-1}\\
        \vdots&\vdots&\ddots&\vdots&\vdots\\
        b_{-n+1}&b_{-n+2}&\hdots&b_{n}&b_{n+1}
    \end{pmatrix},
    \]
    \setstretch{1.225}%
    where~$m,n\in\mathbb{Z}_{>0}$. This matrix is totally nonnegative, and $a_m\ne0$, $b_m\ne0$
    when~$m$ is big enough (see Lemma~\ref{lemma:ann_convergence_term}). Divide the first row
    by~$a_{m}$ and let~$m\to+\infty$. Then Lemma~\ref{lemma:ann_convergence} gives the relation
    \[
    \frac{a_{m+n}}{a_{m}}
    =\frac{a_{m+1}}{a_{m}}\cdot\frac{a_{m+2}}{a_{m+1}}\cdots\frac{a_{m+n}}{a_{m+n-1}}
    \to R^{-n}
    \]
    implying that
    \[
    \begin{pmatrix}
        1&\frac{a_{m+1}}{a_{m}}&\hdots&\frac{a_{m+2n-1}}{a_{m}}&\frac{a_{m+2n}}{a_{m}}\\
        a_0&a_{1}&\hdots&a_{2n-1}&a_{2n}\\
        b_0&b_{1}&\hdots&b_{2n-1}&b_{2n}\\
        a_{-1}&a_0&\hdots&a_{2n-2}&a_{2n-1}\\
        b_{-1}&b_0&\hdots&b_{2n-2}&b_{2n-1}\\
        \vdots&\vdots&\ddots&\vdots&\vdots\\
        b_{-n+1}&b_{-n+2}&\hdots&b_{n}&b_{n+1}
    \end{pmatrix}
    \to
    \begin{pmatrix}
        1&R^{-1}&\hdots&R^{-2n+1}&R^{-2n}\\
        a_0&a_{1}&\hdots&a_{2n-1}&a_{2n}\\
        b_0&b_{1}&\hdots&b_{2n-1}&b_{2n}\\
        a_{-1}&a_0&\hdots&a_{2n-2}&a_{2n-1}\\
        b_{-1}&b_0&\hdots&b_{2n-2}&b_{2n-1}\\
        \vdots&\vdots&\ddots&\vdots&\vdots\\
        b_{-n+1}&b_{-n+2}&\hdots&b_{n}&b_{n+1}
    \end{pmatrix}\eqqcolon M,
    \]
    where the convergence is entry-wise. The matrix~$M$ is totally nonnegative as an entry-wise
    limit of totally nonnegative matrices. Subtracting columns in~$M$ then gives
    \[
    M_1\coloneqq
    \begin{pmatrix}
        1&0&\hdots&0&0\\
        a_0&a_{1}-a_0R^{-1}
               &\hdots&a_{2n-1}-a_{2n-2}R^{-1}&a_{2n}-a_{2n-1}R^{-1}\\
        b_0&b_{1}-b_0R^{-1}
               &\hdots&b_{2n-1}-b_{2n-2}R^{-1}&b_{2n}-b_{2n-1}R^{-1}\\
        a_{-1}&a_0-a_{-1}R^{-1}
               &\hdots&a_{2n-2}-a_{2n-3}R^{-1}&a_{2n-1}-a_{2n-2}R^{-1}\\
        b_{-1}&b_0-b_{-1}R^{-1}
               &\hdots&b_{2n-2}-b_{2n-3}R^{-1}&b_{2n-1}-b_{2n-2}R^{-1}\\
        \vdots&\vdots&\ddots&\vdots&\vdots\\
        b_{-n+1}&b_{-n+2}-b_{-n+1}R^{-1}
               &\hdots&b_{n-2}-b_{n-3}R^{-1}&b_{n-1}-b_{n-2}R^{-1}
    \end{pmatrix}.
    \]
    If the matrix~$M$ is totally nonnegative, then by Theorem~\ref{th:whitney} the matrix~$M_1$
    is also totally nonnegative%
    \footnote{We apply Theorem~\ref{th:whitney} to the transpose of the matrices~$M$ and~$M_1$;
        then we use the fact that the transposition does not affect the total nonnegativity.} %
    . Now note that
    \[
    \begin{aligned}
        p_1(z)&=\left(1-\frac{z}{R}\right)\sum_{m=-\infty}^\infty a_{m}z^m
               = \sum_{m=-\infty}^\infty a_{m+1}z^{m+1} - \sum_{m=-\infty}^\infty \frac{a_{m}}{R}z^{m+1}
               = \sum_{m=-\infty}^\infty \big(a_{m+1}-a_{m}R^{-1}\big)z^{m+1}
               \\
               \text{and}\quad
        q_1(z)&=\left(1-\frac{z}{R}\right)\sum_{m=-\infty}^\infty b_{m}z^m
               =\sum_{m=-\infty}^\infty (b_{m+1}-b_{m}R^{-1})z^{m+1},
    \end{aligned}
    \]
    so any submatrix of~$H(p_1,q_1)$ has only nonnegative minors, and hence $H(p_1,q_1)$ is
    itself totally nonnegative.

    Analogously, suppose that~$r>0$. Then Lemmata~\ref{lemma:ann_convergence_term}
    and~\ref{lemma:ann_convergence} imply~$\frac{b_{m}}{b_{m+n}}\to r^{n}$ as $m\to-\infty$, and
    therefore
    \[
    \begin{pmatrix}
        a_{0}&a_{1}&\hdots&a_{2n-1}&a_{2n}\\
        b_{0}&b_{1}&\hdots&b_{2n-1}&b_{2n}\\
        a_{-1}&a_{0}&\hdots&a_{2n-2}&a_{2n-1}\\
        b_{-1}&b_{0}&\hdots&b_{2n-2}&b_{2n-1}\\
        \vdots&\vdots&\ddots&\vdots&\vdots\\
        b_{-n+1}&b_{-n+2}&\hdots&b_{n}&b_{n+1}\\
        \frac{b_{m}}{b_{m+2n}}&\frac{b_{m+1}}{b_{m+2n}}&\hdots&\frac{b_{m+2n-1}}{b_{m+2n}}&1
    \end{pmatrix}
    \to
    \begin{pmatrix}
        a_{0}&a_{1}&\hdots&a_{2n-1}&a_{2n}\\
        b_{0}&b_{1}&\hdots&b_{2n-1}&b_{2n}\\
        a_{-1}&a_{0}&\hdots&a_{2n-2}&a_{2n-1}\\
        b_{-1}&b_{0}&\hdots&b_{2n-2}&b_{2n-1}\\
        \vdots&\vdots&\ddots&\vdots&\vdots\\
        b_{-n+1}&b_{-n+2}&\hdots&b_{n}&b_{n+1}\\
        r^{2n}&r^{2n-1}&\hdots&r&1\\
    \end{pmatrix},
    \]
    \setstretch{1.2}%
    where the convergence is entry-wise. Subtracting columns gives
    \[
    M_2\coloneqq
    \begin{pmatrix}
        a_0-ra_{1}&a_{1}-ra_{2}&\hdots&a_{2n-1}-ra_{2n}&a_{2n}\\
        b_0-rb_{1}&b_{1}-rb_{2}&\hdots&b_{2n-1}-rb_{2n}&b_{2n}\\
        a_{-1}-ra_0&a_0-ra_{1}&\hdots&a_{2n-2}-ra_{2n-1}&a_{2n-1}\\
        b_{-1}-rb_0&b_0-rb_{1}&\hdots&b_{2n-2}-rb_{2n-1}&b_{2n-1}\\
        \vdots&\vdots&\ddots&\vdots&\vdots\\
        b_{-n+1}-rb_{-n+2}&b_{-n+2}-rb_{-n+3}&\hdots&b_{n}-rb_{n+1}&b_{n+1}\\
        0&0&\hdots&0&1\\
    \end{pmatrix}.
    \]
    Since the matrix~$H(p,q)$ is totally nonnegative, the matrix~$M_2$ is also totally
    nonnegative by Theorem~\ref{th:whitney} and Lemma~\ref{lemma:TNN_order}. Crossing out the
    last column and row from~$M_2$ gives a submatrix of~$H(p_2,q_2)$, because
    \[
    p_2(z)\coloneqq \left(1-\frac{r}{z}\right)p(z)=\sum_{k=-\infty}^{+\infty}(a_k-ra_{k+1})z^k
    \an
    q_2(z)\coloneqq \left(1-\frac{r}{z}\right)q(z)=\sum_{k=-\infty}^{+\infty}(b_k-rb_{k+1})z^k.
    \]
    The integer~$n\ge 0$ is arbitrary, and thus the whole matrix of~$H(p_2,q_2)$ is totally
    nonnegative.
\end{proof}
\setstretch{1.2}%
\begin{lemma}
    \label{lemma:comm_poles}
    If~$H(p,q)$ is totally nonnegative and has a nonzero minor of order~$\ge 2$, then the
    series~$p(z)$ and~$q(z)$ can be represented as~$p(z)=p_*(z)g(z)$ and~$q(z)=q_*(z) g(z)$,
    respectively. Here $g(z)$ denotes a function of the form~\eqref{eq:funct_gen_dtps} and the
    matrix~$H(p_*,q_*)$ is totally nonnegative. Moreover, both~$p_*(z)$
    and~$q_*(z)$ can be represented as the products
    \begin{equation}\label{eq:pq_form_0}
        \begin{aligned}
            p_*(z)&=C_1 z^je^{Az+\frac{A_0}z} \prod_{\nu>0} \left(1+\frac{z}{\alpha_\nu}\right)
            \prod_{\nu<0} \left(1+\frac{z^{-1}}{\alpha_\nu}\right)
            \an\\
            q_*(z)&=C_2 z^ke^{Az+\frac{A_0}z} \prod_{\nu>0} \left(1+\frac{z}{\beta_\nu}\right)
            \prod_{\nu<0} \left(1+\frac{z^{-1}}{\beta_\nu}\right),
        \end{aligned}
    \end{equation}
    where $C_1,C_2,A,A_0\ge 0$, the exponents~$j,k\in \mathbb{Z}_{\ge0}$ and~$\alpha_\nu,\beta_\nu>0$
    for all~$\nu\in\mathbb{Z}_{\ne0}$.
\end{lemma}
\begin{proof}
    If one of the series is trivial, then this theorem is equivalent to Theorem~\ref{th:E-AESW};
    therefore, we suppose below in this proof that both~$p(z)$ and~$q(z)$ are not trivial.
    
    By Lemma~\ref{lemma:ann_convergence}, $p(z)$ and~$q(z)$ converge in the same
    annulus~$r<|z|<R$, that is $\left||z|-\rho\right|<(R-r)/2$, where~$\rho\coloneqq(R+r)/2$. The
    annulus is not empty, because there exists a nonzero minor of~$H(p,q)$ of order~$\ge 2$.
    Since the matrices~$T(p)$ and~$T(q)$ are totally nonnegative as submatrices of~$H(p,q)$,
    both series~$p(z)$ and~$q(z)$ converge to functions of the form~\eqref{eq:funct_gen_dtps} by
    Theorem~\ref{th:E-AESW}. In particular, the poles of these functions are positive and can
    condense only at~$z=0$ or infinity. Let us enumerate all their \emph{common} poles
    as~$\gamma_0,\gamma_1,\dots,\gamma_N$ excluding a possible pole at the origin, so that
    \[
    1<\max\left\{\frac{\rho}{\gamma_i},\frac{\gamma_i}{\rho}\right\}%
    \le%
    \max\left\{\frac{\rho}{\gamma_{i+1}},\frac{\gamma_{i+1}}{\rho}\right\}\]
    for all~$i<N\le\infty$ and each pole occurs only once. Denote
    \[
    p_n(z)\colonequals p(z)\prod_{i=0}^n E_i^{m_i}(z)
    \an
    q_n(z)\colonequals q(z)\prod_{i=0}^n E_i^{m_i}(z),
    \ww
    E_i(z)\colonequals
    \begin{cases}
        1-\frac{z}{\gamma_i}&\text{if }\gamma_i>\rho,\\
        1-\frac{\gamma_i}{z}&\text{if }\gamma_i<\rho
    \end{cases}
    \]
    and~$m_i$ stands for the order%
    \footnote{In other words, $m_i$ is the minimal number, such that at least one of the
        functions $p(z)E_i^{m_i}(z)$ and $q(z)E_i^{m_i}(z)$ is regular at the
        point~$z=\gamma_i$.} %
    of the pole~$\gamma_i$.

    \setstretch{1.25}%
    Let~$n=0$. On the one hand, both~$p_0(z)$ and~$q_0(z)$ represent functions of the
    form~\eqref{eq:funct_gen_dtps}, at least one of which has no pole at~$\gamma_0\in\{r,R\}$.
    On the other hand, these series converge in the same annulus by
    Lemma~\ref{lemma:remove_pole} since the matrix~$H(p_0,q_0)$ is totally nonnegative.
    Consequently, neither of~$p_0(z)$ and~$q_0(z)$ has a pole at~$\gamma_0$, and $r<|z|<R$ is
    strictly nested in the annulus of convergence of these series. By induction
    on~$n=0,1,\dots,N$ we obtain that the matrix~$H(p_n,q_n)$ is totally nonnegative for
    each~$n$, and hence the orders of the pole~$\gamma_n$ of~$p(z)$ and of~$q(z)$ coincide.

    Both~$p(z)$ and~$q(z)$ can be represented as in~\eqref{eq:funct_gen_dtps}, so the
    product~$\prod_{i=0}^N E_i^{m_i}(z)$ converges in~$\mathbb C\setminus\{0\}$ locally
    uniformly when~$N=\infty$.
    Therefore,~$p_n(z)\to p_*(z)\colonequals p(z)\prod_{i=0}^N E_i^{m_i}(z)$
    and~$q_n(z)\to q_*(z)\colonequals q(z)\prod_{i=0}^N E_i^{m_i}(z)$, and the Laurent
    coefficients converge as well. Moreover, both functions~$p_*(z)$ and~$q_*(z)$ are
    holomorphic in~$\mathbb{C}\setminus\{0\}$: they have the form
    \begin{equation}\label{eq:pq_form0}
        \begin{aligned}
            p_*(z)&= C_1z^je^{Az+\frac{A_0}z}\cdot
                \prod_{\nu>0} \left(1+\frac{z}{\alpha_\nu}\right)
                \prod_{\nu<0} \left(1+\frac{z^{-1}}{\alpha_\nu}\right)
            \an\\
            q_*(z)&= C_2z^ke^{Bz+\frac{B_0}z}\cdot
                \prod_{\mu>0} \left(1+\frac{z}{\beta_\mu}\right)
                \prod_{\mu<0} \left(1+\frac{z^{-1}}{\beta_\mu}\right),
        \end{aligned}
    \end{equation}
    where $A,A_0,B,B_0,C_1,C_2\ge 0$; \ $j,k\in \mathbb{Z}$ and~$\alpha_\nu,\beta_\mu>0$ for
    all~$\nu,\mu$. The corresponding Hurwitz-type matrix~$H(p_*,q_*)$ is totally nonnegative as
    the entry-wise limit of the matrices~$H(p_n,q_n)$ as~$n\to N$. If~$p_*(z)$ is a
    Laurent polynomial, then~$q_*(z)$ is also a Laurent polynomial by
    Lemma~\ref{lemma:ann_convergence_term} and both~$p_*(z)$ and~$q_*(z)$ have
    only negative zeros by Theorem~\ref{th:E-AESW}; in this case the assertion of the lemma
    holds true with~$g(z)=\prod_{i=0}^NE_i^{-m_i}(z)$, so below we suppose
    that~$p_*(z)$ has infinitely many nonzero coefficients.

    To keep the notation, assume that~$p(z)=p_*(z)$ and~$q(z)=q_*(z)$, so that
    both series converge in~$\mathbb{C}\setminus\{0\}$. According to
    Lemma~\ref{lemma:ann_convergence_term}, we can fix an integer~$k$ such
    that~$a_{k-1},a_k,a_{k+1}\ne 0$, and thus~$b_k,b_{k+1}\ne 0$. The conditions
    \[
    0<p(x)=\max_{|z|=x}|p(x)|<\infty
    \an
    0<q(x)=\max_{|z|=x}|q(z)|<\infty
    \]
    hold for~$x>0$ since both series~$p(x)$ and~$q(x)$ have only nonnegative coefficients. From
    the total nonnegativity of~$H(p,q)$ we obtain the
    inequalities~$a_{n-1}\frac{b_{k+1}}{a_k}\ge b_n$ and~$b_n\ge a_n\frac{b_k}{a_k}$ for
    all~$n>k$, which are giving
    \[
    x\frac{b_{k+1}}{a_k}\sum_{n=k}^\infty a_nx^n
    \ge \sum_{n=k+1}^\infty b_nx^n
    \ge \frac{b_k}{a_k}\sum_{n=k+1}^\infty a_nx^n
    \quad\text{for}\quad x>0.
    \]
    Analogously, the inequalities~$a_n\frac{b_k}{a_k}\ge b_n$
    and~$b_n\ge a_{n-1}\frac{b_{k+1}}{a_k}$ for~$n\le k$ give
    \[
    \frac{b_k}{a_k}\sum_{n=-\infty}^{k} a_nx^n
    \ge \sum_{n=-\infty}^{k} b_nx^n
    \ge x\frac{b_{k+1}}{a_k}\sum_{n=-\infty}^{k-1} a_nx^n
    \quad\text{for}\quad x>0.
    \]
    Summing up then yields
    \[
    \begin{aligned}
        x\frac{b_{k+1}}{a_k}p(x) &\ge q(x) + O(x^{k}) \ge\frac{b_k}{a_k}p(x) &\text{and}&&
        \frac{p(x)}{x^k}&\ge a_{k+1}x\to+\infty&\text{as}&& x&\to+\infty,
        \\
        x\frac{b_{k+1}}{a_k}p(x)&\le q(x)+O(x^{k+1}) \le\frac{b_k}{a_k}p(x) &\text{and}&&
        \frac{xp(x)}{x^{k+1}}&\ge \frac{a_{k-1}}x\to+\infty &\text{as}&& x&\to+0.
    \end{aligned}
    \]
    In other words, the big-O terms are neglectable and we can write
    \begin{equation}\label{eq:qp_lim_x_inf}
    \varlimsup_{x\to+\infty}\frac{q(x)}{xp(x)}=
    \varlimsup_{x\to+\infty}
    \frac{C_2x^ke^{Bz+\frac{B_0}x}\cdot
        \prod_{\mu>0} \left(1+\frac{x}{\beta_\mu}\right)
        \prod_{\mu<0} \left(1+\frac{x^{-1}}{\beta_\mu}\right)}
    {C_1x^{j+1}e^{Ax+\frac{A_0}x}\cdot
        \prod_{\nu>0} \left(1+\frac{x}{\alpha_\nu}\right)
        \prod_{\nu<0} \left(1+\frac{x^{-1}}{\alpha_\nu}\right)}
    \le\frac{b_{k+1}}{a_k}
    \an
    \varliminf_{x\to+\infty}\frac{q(x)}{p(x)}\ge\frac{b_k}{a_k}
    \end{equation}
    \setstretch{1.2}%
    with positive~$A,B,A_0,B_0$. Recall that~$\frac{b_k}{a_k}>0$, which implies~$B=A$ since
    otherwise the dominant infinite products~$\prod_{\mu>0} \left(1+\frac{x}{\beta_\mu}\right)$
    and~$\prod_{\nu>0} \left(1+\frac{x}{\alpha_\nu}\right)$ are neglectable with respect to the
    exponential term~$e^{(B-A)x}$ as~$x\to\infty$ by Theorem~\ref{th:canonical_product}.
    Analogously,
    \begin{equation}\label{eq:qp_lim_x_0}
    \varliminf_{x\to+0}\frac{q(x)}{p(x)}=
    \varliminf_{x\to+0}
    \frac{C_2x^ke^{\frac{B_0}x}\cdot
        \prod_{\mu>0} \left(1+\frac{x}{\beta_\mu}\right)
        \prod_{\mu<0} \left(1+\frac{x^{-1}}{\beta_\mu}\right)}
    {C_1z^{j}e^{\frac{A_0}z}\cdot
        \prod_{\nu>0} \left(1+\frac{x}{\alpha_\nu}\right)
        \prod_{\nu<0} \left(1+\frac{x^{-1}}{\alpha_\nu}\right)}
    \le\frac{b_k}{a_k}
    \an
    \varlimsup_{x\to+0}\frac{q(x)}{xp(x)}\ge\frac{b_{k+1}}{a_k}
    \end{equation}
    which implies~$B_0=A_0$ due to~$\frac{b_{k+1}}{a_k}>0$. Therefore, to finally prove the
    lemma it is enough to take~\(g(z)=\prod_{i=0}^NE_i^{-m_i}(z)\).
\end{proof}
\setstretch{1.2}%

\section{\texorpdfstring{$\mathcal{S}$}{S}-functions}
\begin{lemma}\label{lemma:product_is_S}
    The product
    \begin{equation}\label{eq:order_alphas_betas}
        \begin{gathered}
            C \frac{\prod_{\mu>0} \left(1+\dfrac{z}{\beta_\mu}\right)}
            {\prod_{\nu>0} \left(1+\dfrac{z}{\alpha_\nu}\right)}
            \frac{\prod_{\mu<0}\left(1+\dfrac{z^{-1}}{\beta_\mu}\right)}
            {\prod_{\nu<0}\left(1+\dfrac{z^{-1}}{\alpha_\nu}\right)},
            \quad\text{where $C>0$, and the numbers}
            \\
            0<\cdots<\alpha_{-2}^{-1}<\beta_{-1}^{-1}<\alpha_{-1}^{-1}
            <\beta_{1}<\alpha_{1}<\beta_{2}<\alpha_2<\cdots
        \end{gathered}
    \end{equation}
    satisfy~$\sum_{\nu\ne 0} \big(\alpha_\nu^{-1} + \beta_\nu^{-1}\big) <\infty$, determines
    an~$\mathcal S$-function. Analogously,
    \begin{equation}\label{eq:order_alphas_betas_mer}
        C
        \frac{z+\beta_0}{z+\alpha_0}\cdot
        \frac{\prod_{\mu>0} \left(1+\dfrac{z}{\beta_\mu}\right)}
        {\prod_{\nu>0} \left(1+\dfrac{z}{\alpha_\nu}\right)}
        ,
    \end{equation}
    \setstretch{1.2}%
    where~$C\ge 0$ and the numbers~$0\le\beta_{0}<\alpha_{0}<\beta_{1}<\alpha_{1}<\cdots$
    satisfy~$\sum_{\nu>0} \big(\alpha_\nu^{-1} + \beta_\nu^{-1}\big) <\infty$ is a
    meromorphic $\mathcal S$\nobreakdash-function. Products over~$\mu$ and~$\nu$
    in~\eqref{eq:order_alphas_betas} or~\eqref{eq:order_alphas_betas_mer} can be terminating, in
    which case the numerator and the denominator retain to have interlacing zeros.
\end{lemma}
\begin{proof}
    Suppose that~$F(z)$ has the form~\eqref{eq:order_alphas_betas} and denote
    \begin{equation}\label{eq:F_is_limit}
        F_n(z)\coloneqq C\frac{q_n(z)}{p_n(z)},\ww
        q_n(z)=\prod_{\nu=1}^{n} \left(1+\frac{z}{\beta_\nu}\right)
        \left( 1+\frac{z^{-1}}{\beta_{-\nu}}\right)
        ,\quad
        p_n(z)=\prod_{\nu=1}^{n} \left(1+\frac{z}{\alpha_\nu}\right)
        \left( 1+\frac{z^{-1}}{\alpha_{-\nu}}\right).
    \end{equation}
    Note that the
    product~$\prod_{\nu=-n}^{-1}\frac{\alpha_\nu}{\beta_\nu}
    =\prod_{\nu=-n}^{-1}\frac{\beta_\nu^{-1}}{\alpha_\nu^{-1}}<1$ is bounded. For
    each~$n\in\mathbb{Z}_{>0}$, the rational function
    \begin{equation}\label{eq:F_is_limit_ML}
        F_n(z)
        =
        C\cdot
        \prod_{\nu=-n}^{-1}\frac{\alpha_\nu}{\beta_\nu}
        +\sum_{\nu=1}^n\left(\frac{A_{\nu,n} z}{z+\alpha_\nu}
            +\frac{A_{-\nu,n} z}{z+\frac{1}{\alpha_{-\nu}}}
        \right),
        \ww
        A_{\nu,n}=\left.C\frac{q_n(z)}{z p'_n(z)}\right|_{z=-\alpha_\nu^{\sign\nu}}>0,
    \end{equation}
    is an~$\mathcal S$-function as each of its partial fractions is such. The
    condition~$\sum_{\nu\ne 0} \big(\alpha_\nu^{-1} + \beta_\nu^{-1}\big) <\infty$ implies the
    locally uniform convergence of each product in~\eqref{eq:order_alphas_betas} (see
    Theorem~\ref{th:canonical_product}) and, therefore, of the numerator~$q_n(z)$ and the
    denominator~$p_n(z)$ as~$n\to\infty$. Since the denominator is nonzero for~$z\not\le0$, the
    function~$F(z)$ is the limit of~$F_n(z)$ as~$n\to\infty$ uniform on compact subsets
    of~$\mathbb{C}\setminus(-\infty,0]$. Moreover,
    \[\Im F(z)\cdot\Im z=\lim_{n\to\infty}F_n(z)\cdot\Im z\ge 0;\]
    the inequality is strict outside the real line due to the maximum principle for the harmonic
    function~$\Im F(z)$.

    The assertion that the expression~\eqref{eq:order_alphas_betas_mer} represents an
    $\mathcal S$-function follows by omitting from~\eqref{eq:F_is_limit_ML} terms that
    correspond to absent poles.
\end{proof}
\begin{lemma}\label{lemma:prop_S}
    Given~$F(z)$ of the form~\eqref{eq:order_alphas_betas} or~\eqref{eq:order_alphas_betas_mer},
    the function~$z^pF(z)$ with any integer~$p\ne0$ cannot be a mapping of~$\mathbb{C}_+$ into
    itself provided that it is not equal identically to~$Cz$ or~$C$; the
    function~$\frac{z}{F(z)}$ is an $\mathcal{S}$-function of the
    form~\eqref{eq:order_alphas_betas} or~\eqref{eq:order_alphas_betas_mer}.
\end{lemma}
\begin{proof}
    Let~$-\alpha_\nu$ be an arbitrary pole of~$F(z)$, then~\eqref{eq:F_is_limit_ML} implies
    that~$F(z)\sim \frac{-A_\nu\alpha_\nu}{z+\alpha_\nu}$ for~$z$ close enough to~$-\alpha_\nu$,
    where~$A_\nu=\lim_{n\to\infty}A_{\nu,n}>0$. If~$z^pF(z)$ also maps the upper half of the
    complex plane into itself, then~$(-\alpha_\nu)^{p+1}A_\nu<0$, and hence~$p\ne0$ must be an even
    number. Nevertheless, if~$z_*=\exp{\frac{\pi}{\left|p\right|} i}$, then~$\Im z_*>0$ and
    $\Im z_*^pF(z_*) = -\Im F(z_*)<0$.

    For the second part of the lemma, it is enough to note that the reciprocal of the
    product~\eqref{eq:order_alphas_betas} can be expressed as
    \begin{equation}\label{eq:recipr_to_S}
        \frac 1{F(z)}= C \frac
        {\prod_{\nu>0} \left(1+\frac{z}{\alpha_\nu}\right)}
        {\prod_{\mu>0} \left(1+\frac{z}{\beta_\mu}\right)}
        \frac{\prod_{\nu<0} \left(1+\frac{z^{-1}}{\alpha_\nu}\right)}
        {\prod_{\mu<0}\left(1+\frac{z^{-1}}{\beta_\mu}\right)}
        =
        \frac C{z\alpha_{-1}} \frac
        {\left(z\alpha_{-1}+1\right)\prod_{\nu>0} \left(1+\frac{z}{\alpha_\nu}\right)}
        {\prod_{\mu>0} \left(1+\frac{z}{\beta_\mu}\right)}
        \frac{\prod_{\nu<-1}\left(1+\frac{z^{-1}}{\alpha_\nu}\right)}
        {\prod_{\mu<0} \left(1+\frac{z^{-1}}{\beta_\mu}\right)},
    \end{equation}
    so the relabelling the~$\beta_\mu\mapsto \widetilde\alpha_\mu$ for all~$\mu\ne0$;
    $\alpha_{-1}\mapsto\widetilde\beta_1^{-1}$ and~$\alpha_\nu\mapsto \widetilde\beta_{\nu-1}$ for
    all~$\nu\notin\{0,1\}$ yields that~$\frac z{F(z)}$ has the
    form~\eqref{eq:order_alphas_betas}. An analogous reasoning works for the reciprocal
    of~\eqref{eq:order_alphas_betas_mer}.
\end{proof}
\begin{proof}[Proof of Lemma~\ref{lemma:prop_S1}]
    On account of Lemma~\ref{lemma:product_is_S}, it is enough to prove that the function~$F(z)$
    is an $\mathcal S$-function only if has the form~\eqref{eq:order_alphas_betas}
    or~\eqref{eq:order_alphas_betas_mer}.  Since~$F(z)$ is regular for~$z>0$, the poles
    of~$p(z)$ and~$q(z)$ coincide and have same orders. The Carath\'eodory inequality implies
    that the absolute value of a function~$\phi(z)$ mapping the upper half of the complex plane
    into itself satisfies (see e.g.~\cite[p.~71]{ChebMei})
    \begin{equation}\label{eq:caratheodory_cr}
        \frac{2}{c\rho}\le |\phi(\rho e^{i\frac {\pi}{6}})|\le 2c\rho,
    \end{equation}
    where~$c>0$ is an appropriate constant and~$\rho>1$. All infinite products
    in~\eqref{eq:funct_gen_dtps} cannot grow or decrease at an exponential rate in~$|z|$ (see
    Theorem~\ref{th:canonical_product}). Thus, if~$F(z)$ were able to have an exponential factor
    of the form~$e^{\pm Az}$ with~$A>0$, then
    necessarily~$|F(\rho e^{i\frac {\pi}{6}})| = |F(\rho\frac{\sqrt 3}{2} + i\frac \rho2)| \sim
    e^{\pm \frac{\sqrt 3}{2}A\rho}$
    as~$\rho\gg 1$, which contradicts~\eqref{eq:caratheodory_cr}. The
    function~$\frac{1}{F(1/z)}$ maps the upper half of the complex plane into itself; thus, it
    satisfies the inequality~\eqref{eq:caratheodory_cr}, which implies that an exponential
    factor of the form~$e^{\pm A_0/z}$ in~$F(z)$ with~$A_0>0$ is absent. Summing up, the
    exponential factors in~$p(z)$ and~$q(z)$ must be the same.

    Zeros and poles of~$\frac{q(z)}{p(z)}$ interlace, because all its poles are simple and the
    residues are negative. Unless~$F(z)$ is meromorphic, the order of zeros can be made as
    in~\eqref{eq:order_alphas_betas} by taking out some power of~$z$,
    cf.~\eqref{eq:recipr_to_S}. The resulting power of the factor~$z$ in the representation
    of~$F(z)$ then must be~$0$ by Lemma~\ref{lemma:prop_S}. The case of meromorphic~$F(z)$ is
    more straightforward: no poles of~$F(z)$ can appear between its maximal zero and the origin,
    otherwise~$F(0+)<0$ due to negativity of residues.
\end{proof}
\section{Total nonnegativity and interlacing zeros}
\begin{lemma}\label{lemma:T_via_HH}
    If the matrix~$H(p,q)$ is totally nonnegative and~$\widetilde p(z)\coloneqq zp(z)$, then for
    arbitrarily taken nonnegative numbers~$A$ and~$B$ both matrices~$T(Ap+Bq)$
    and~$T(Aq+B\widetilde p)$ are totally nonnegative.
\end{lemma}
\begin{proof}
    Observe that
    \begin{equation}\label{eq:T_via_HH}
        T(Ap+Bq) = H^{\textsf{T}}(A,B)\,H(p,q)
        \an
        T(Aq+B\widetilde p) = H^{\textsf{T}}(A,B)\,H(q,\widetilde p),
    \end{equation}
    where the auxiliary totally nonnegative matrix~$H^{\textsf{T}}(A,B)$ is the transpose
    of~$H(A,B)$:
    \[
    H^{\textsf{T}}(A,B)=
    \begin{pmatrix}
        \ddots & \vdots & \vdots & \vdots & \vdots & \vdots & \vdots & \iddots\\
        \hdots & A & B & 0 & 0 & 0 & 0 & \hdots\\
        \hdots & 0 & 0 & A & B & 0 & 0 & \hdots\\
        \hdots & 0 & 0 & 0 & 0 & A & B & \hdots\\
        \iddots & \vdots & \vdots & \vdots & \vdots & \vdots & \vdots & \ddots
    \end{pmatrix}
    =\big(h_{ij}\big)_{i,j=-\infty}^{\infty},\ww
    h_{ij}=\begin{cases}
        A&\text{if } j=2i-1,\\
        B&\text{if } j=2i,\\
        0&\text{otherwise}.
    \end{cases}
    \]
    By Lemma~\ref{lemma:H_zpq_order}, the matrix~$H(q,\widetilde p)$ is totally nonnegative
    whenever~$H(p,q)$ is totally nonnegative. Therefore, applying the Cauchy-Binet formula to
    the expressions~\eqref{eq:T_via_HH} yields that all minors of the matrices~$T(Ap+Bq)$
    and~$T(Aq+B\widetilde p)$ must be nonnegative.
\end{proof}

\begin{lemma}\label{lemma:H_TNN_F_has_form}
    If~$H(p,q)$ is totally nonnegative and has a nonzero minor of order~$\ge 2$ and,
    additionally,~$p(z)\not\equiv 0$, then the ratio~$F(z)\coloneqq\frac{q(z)}{p(z)}$ of
    functions represented by the series~$p(z)$ and~$q(z)$ has one of
    forms~\eqref{eq:order_alphas_betas} or~\eqref{eq:order_alphas_betas_mer}
    or~$C\cdot(z+\beta_0)$ with~$C,\beta_0\ge 0$.
\end{lemma}
\begin{proof}
    Lemma~\ref{lemma:comm_poles} establishes that~$p(z)=p_*(z)g(z)$ and~$q(z)=p_*(z)g(z)$, where
    the series~$p_*(z)$ and~$q_*(z)$ converge to functions of the form~\eqref{eq:pq_form_0}
    and~$g(z)$ can be represented as in~\eqref{eq:funct_gen_dtps}. Furthermore, the
    matrix~$H(p_*,q_*)$ is totally nonnegative. Since the common factor~$g(z)$ does not affect
    the ratio~$F(z)$, we assume that~$p(z)=p_*(z)=\sum_{k=-\infty}^\infty a_kz^k$
    and~$q(z)=q_*(z)=\sum_{k=-\infty}^\infty b_kz^k$ without loss of generality. In particular,
    both series~$p(z)$ and~$q(z)$ converge to functions for all~$z\ne 0$ and we keep the same
    designations for the functions. Now note that the totally nonnegative Toeplitz matrices in
    Lemma~\ref{lemma:T_via_HH} are constructed from the coefficients of the series~$Ap(z)+Bq(z)$
    and~$Aq(z)+Bzp(z)$, and hence the limits of these series have the
    form~\eqref{eq:funct_gen_dtps}. In particular, all zeros of the functions~$Ap(z)+Bq(z)$
    and~$Aq(z)+Bzp(z)$ lie in~$(-\infty,0]$; consequently, they also can be expressed as
    in~$\eqref{eq:pq_form_0}$ due to absence of poles.

    If~$z_0$ is such that~$F(z_0)=\frac{q(z_0)}{p(z_0)}\le 0$, then~$\phi(z_0;-F(z_0))=0$,
    where we put~$\phi(z;A)\coloneqq Ap(z)+q(z)$. Since for each~$A\ge 0$ the
    function~$\phi(z;A)$ does not vanish outside~$(-\infty,0]$, the inequality~$z_0\le 0$ must
    be true. Analogously, if~$z_1$ is such
    that~$\frac{F(z_1)}{z_1}=\frac{q(z_1)}{z_1p(z_1)}\le 0$, then~$\psi(z_1;-F(z_1)/z_1)=0$,
    where~$\psi(z;B)\colonequals Bzp(z)+q(z)$. Since for each~$B\ge 0$ the
    function~$\psi(z;B)$ is nonzero outside~$(-\infty,0]$, we obtain~$z_1\le 0$.

    \begin{fact}[{Details can be found in \emph{e.g.}~\cite[p.~19]{Duren2004}}]\label{fact:A}
        Let~$h(z)$ be a real function holomorphic in a neighbourhood of a real point $x$ and
        such that~$h(z)\le 0$ for a complex~$z$ implies~$z\le 0$. Then the
        expression~$h(z)-h(x)$ has a zero at~$x$ of some multiplicity~$r\ge 1$.
        Therefore,~$h(z)-h(x)\sim (z-x)^{r}$ as~$z$ is close to~$x$ in a small enough
        neighbourhood of~$x$, and we have~$\Im h(z)=0$ on the union of~$r$ arcs meeting in this
        neighbourhood only at~$x$; one of these arcs is a subinterval of the real line due to
        the reality of~$h(z)$. Furthermore, the half of (if~$r$ is even) or all (if~$r$ is odd)
        the arcs contain an interval where~$h(z)\le h(x)$.
        \emph{In particular, the condition~$h(x)=0$ implies~$r\le 2$, and the condition~$h(x)<0$
            implies~$r=1$.}
    \end{fact}
    Fact~\ref{fact:A} with~$h(z)\coloneqq F(z)$ implies that~$F(z)$ can have at most double
    zeros. It is possible that the function~$F(z)$ is holomorphic at the origin and equal to
    zero there. The assumption that the point~$x=0$ can be a double zero of~$F(z)$ is
    contradictory: Fact~\ref{fact:A} implies that~$F(z)$ is negative for all real~$z\ne0$ small
    enough, which is impossible for~$z>0$. Suppose that~$x<0$ is a double zero of~$F(z)$, that
    is~$F(x)=F'(x)=0\ne F''(x)$. Then~$F(z)<0$ and, therefore,~$z^{-1}F(z)>0$ for all real~$z$
    in a sufficiently small punctured neighbourhood of~$x$. At the point~$x$, the
    function~$z^{-1}F(z)$ has a double zero:
    \[
    \frac{F(x)}x=\frac{F(x)-xF'(x)}{x^2}=0\ne \frac{2x(F(x)-xF'(x))-x^{3}F''(x)}{x^{4}}.
    \]
    Putting~$h(z)\coloneqq z^{-1}F(z)$ in Fact~\ref{fact:A} then yields a contradiction, since
    the inequality~$z^{-1}F(z)\le0$ must be satisfied for all real~$z$ which are close enough
    to~$x$. Consequently, the only possible case is~$r=1$, that is that all zeros of~$F(z)$ are
    simple. Considering in the same way~$h(z)=\frac{z}{F(z)}$ and~$h(z)=\frac{1}{F(z)}$ shows
    that all poles of~$\frac{F(z)}{z}$ are simple. In particular,~$F(z)$ cannot have a pole at
    the origin.
    
    Now, let us prove that zeros and poles of~$F(z)$ are interlacing. Suppose
    that~$x_1<x_2\le 0$ are two consecutive zeros of the function~$F(z)$, such that the
    interval~$(x_1,x_2)$ contains no poles of~$F(z)$. The ratio~$z^{-1}F(z)$ also vanishes
    at~$x_1$ and~$x_2$ unless~$x_2=0$; therefore, Rolle's theorem gives the
    points~$\xi_1,\xi_2\in(x_1,x_2)$ such
    that~$F'(\xi_1)=\xi_2^{-2}(F(\xi_2)-\xi_2F'(\xi_2))=0$. Let~$h(z)\coloneqq F(z)$
    and~$x\coloneqq\xi_1$ if~$F(\xi_1)<0$, or~$h(z)\coloneqq z^{-1}F(z)$ and~$x\coloneqq\xi_2$
    if~$F(\xi_1)>0$. In the special case~$x_2=0$, the function~$F(z)$ is negative
    in~$(x_1,x_2)$, so we put~$h(z)\coloneqq F(z)$ and denote a zero of~$h'(z)$ in this interval
    by~$x$. Fact~\ref{fact:A} implies~$h'(z)\ne 0$ in the whole interval~$x_1<z<x_2\le 0$
    including~$z=x$, which contradicts to our choice of~$x$. This shows that the
    function~$F(z)$ has at least one pole between each pair of its zeros. The same argumentation
    for~$\frac {z}{F(z)}$ instead of~$F(z)$ yields that~$F(z)$ has a zero between each pair of
    its poles. As a result, zeros and poles of~$F(z)$ are interlacing.

    Recall that the functions~$p(z)$ and~$q(z)$ can be represented as in~\eqref{eq:pq_form_0};
    we can therefore gather all their common zeros and their common exponential
    term~$e^{Az+\frac{A_0}z}$ into a
    function~$g_*(z)=\sum_{n=-\infty}^{\infty}g_nz^n\not\equiv 0$ of the
    form~\eqref{eq:pq_form_0}, such that zeros of~$\frac{q(z)}{g_*(z)}$ coincide with zeros
    of~$F(z)$ and zeros of~$\frac{p(z)}{g_*(z)}$ coincide with poles of~$F(z)$. In the
    case~$q(z)\equiv 0$, the lemma is trivial.
    If~$(\zeta_1z+\zeta_2)p(z)=(\eta_1z+\eta_2)q(z)\not\equiv 0$ with some
    coefficients~$\zeta_1,\zeta_2,\eta_1,\eta_2\ge 0$, then
    \[
    0\le
    \begin{vmatrix}
        a_n& a_{n+1}\\
        b_n& b_{n+1}
    \end{vmatrix}
    =
    \det\left(
    \begin{pmatrix}
        \zeta_2& \zeta_1\\
        \eta_2& \eta_1
    \end{pmatrix}
    \cdot
    \begin{pmatrix}
        g_n& g_{n+1}\\
        g_{n-1}& g_{n}
    \end{pmatrix}\right)
    =
    \begin{vmatrix}
        \zeta_2& \zeta_1\\
        \eta_2& \eta_1
    \end{vmatrix}
    \cdot
    \begin{vmatrix}
        g_n& g_{n+1}\\
        g_{n-1}& g_{n}
    \end{vmatrix}.
    \]
    Since the Laurent series of~$g_*(z)$ converges is any annulus centered at the origin, there
    exists an integer~$n$ such that~$g_n^2>g_{n-1}g_{n+1}$ by
    Lemma~\ref{lemma:ann_tnn_zero_minor}. Therefore, the coefficients of the
    function~$F(z)=\frac{\eta_1z+\eta_2}{\zeta_1z+\zeta_2}$
    satisfy~$\zeta_2\eta_1\ge \eta_2\zeta_1$. In the case~$\zeta_2=0$, we necessarily
    have~$\eta_2=0$ due to~$q(z)\not\equiv 0$; if~$\zeta_2\ne 0$
    then~$\alpha_0=\frac{\zeta_2}{\zeta_1}\ge\frac{\eta_2}{\eta_1}=\beta_0\ge 0$, which is proving
    the lemma.

    Let the function~$F(z)$ be meromorphic and not a constant, and
    let~$C_2(z+\beta_0)p(z)\not\equiv C_1(z+\alpha_0)q(z)$. Since zeros and poles of~$F(z)$ are
    interlacing, this implies that both functions~$\frac{q(z)}{g_*(z)}$
    and~$\frac{p(z)}{g_*(z)}$ have zeros:
    \[
    \frac{p(z)}{g_*(z)}=C_1(z+\alpha_0)\prod_{\nu>0} \left(1+\dfrac{z}{\alpha_\nu}\right)
    \an
    \frac{q(z)}{g_*(z)}=C_2(z+\beta_0)\prod_{\nu>0} \left(1+\dfrac{z}{\beta_\mu}\right)
    \]
    where~$0\le\beta_{0}<\beta_{1}<\cdots$ and~$0<\alpha_{0}<\alpha_{1}<\cdots$. For proving
    that~$F(z)$ has the form~\eqref{eq:order_alphas_betas_mer} it is enough to show that the
    chain inequality~$0<\alpha_0<\beta_0<\alpha_1<\beta_1<\dots$ fails to hold. Let this
    inequality hold, then~$p(z)$ has at least two negative zeros, and thus at least three
    nonzero coefficients of~$p(z)$. By Lemma~\ref{lemma:ann_tnn_zeros}, there are at least three
    consecutive nonzero coefficients, say~$a_{i-1},a_{i},a_{i+1}$. Estimating the terms
    according to
    \[
    \frac{1+\frac{x}{\beta_\nu}}{1+\frac{x}{\alpha_\nu}}<1
    \quad\text{for all}\quad x>0,\quad \nu=1,2,\dots
    \]
    in the ratio~$\frac{q(x)\cdot g_*(x)}{g_*(x)\cdot p(x)}$ yields the contradiction
    \[
    \varliminf_{x\to+\infty}F(x)
    \le\frac{C_2}{C_1}\cdot\lim_{x\to+\infty}\frac{x+\beta_0}{x+\alpha_0}
    =\frac{C_2}{C_1}
    <\frac{C_2\beta_0}{C_1\alpha_0}=F(0)\le \frac{b_i}{a_i}\le
    \varliminf_{x\to+\infty}F(x),
    \]
    where the last two inequalities are the first inequality in~\eqref{eq:qp_lim_x_0} and the
    last inequality in~\eqref{eq:qp_lim_x_inf}. Consequently,~$\beta_0<\alpha_0$ and~$F(z)$ has
    the form~\eqref{eq:order_alphas_betas_mer}.

    Now let us consider the remaining case when the function~$F(z)$ has an essential singularity
    at the origin, that is when the products over~$\nu<0$ in the
    representations~\eqref{eq:pq_form_0} of~$p(z)$ and~$q(z)$ have an infinite number of
    distinct terms. Since the distinct zeros of~$p(z)$ and~$q(z)$ are interlacing, we can
    enumerate entries in~$(\alpha_\nu)_{\nu\ne 0}$ and~$(\beta_\nu)_{\nu\ne 0}$ so that the
    inequality~\eqref{eq:order_alphas_betas} is satisfied:
    \[
    \cdots<\alpha_{-2}^{-1}<\beta_{-1}^{-1}<\alpha_{-1}^{-1}
    <\beta_{1}<\alpha_{1}<\beta_{2}<\alpha_2<\cdots,
    \]
    and therefore for all~$x>0$
    \[
    \frac{\alpha_{-1}x}{1+\alpha_{-1}x}
    =\frac{1}{1+\frac{x^{-1}}{\alpha_{-1}}}
    <
    \prod_{\nu=-\infty}^{-1}\frac{1+\frac{x^{-1}}{\beta_\nu}}{
        1+\frac{x^{-1}}{\alpha_\nu}}<1    
    \an
    1<\prod_{\nu>0}\frac{ 1+\frac{x}{\beta_\nu}}{1+\frac{x}{\alpha_\nu}}
    <1+\frac{x}{\beta_1}.
    \]
    Gathering these estimates together yields
    \begin{equation}\label{eq:G_bounded}
        \frac{\alpha_{-1}x}{1+\alpha_{-1}x}
        <G(x)\coloneqq
        \prod_{\nu=-\infty}^{-1}\frac{1+\frac{x^{-1}}{\beta_\nu}}{
            1+\frac{x^{-1}}{\alpha_\nu}}\cdot
        \prod_{\nu>0}\frac{ 1+\frac{x}{\beta_\nu}}{1+\frac{x}{\alpha_\nu}}
        <1+\frac{x}{\beta_1};
    \end{equation}
    in other words, we put~$G(x)=\frac{C_1}{C_2}x^{j-k}F(x)$. Take~$i$ is such
    that~$a_{i-2},a_{i-1},a_{i},a_{i+1}\ne 0$, then~$b_{i-1},b_{i},b_{i+1}\ne 0$ by
    Lemma~\ref{lemma:ann_convergence_term}. The series~$p(z)$ and~$q(z)$ converge to functions
    holomorphic in~$\mathbb{C}\setminus\{0\}$; if~$b_{i}a_{i-1}=b_{i-1}a_{i}$, then
    Lemma~\ref{lemma:H_tnn_zero_minor}~\eqref{item:Hurwitz_minor_c}
    gives~$q(z)=\frac{b_i}{a_i}p(z)$ contradicting to the presence of distinct zeros of~$p(z)$
    and~$q(z)$. Analogously, if~$b_{i}a_{i}= b_{i+1}a_{i-1}$, then Lemma~\ref{lemma:H_zpq_order}
    and Lemma~\ref{lemma:H_tnn_zero_minor}~\eqref{item:Hurwitz_minor_c} give the
    contradiction~$zp(z)=\frac{a_{i-1}}{b_{i}}q(z)$. As a result, the
    inequalities~\eqref{eq:qp_lim_x_inf} and~\eqref{eq:qp_lim_x_0} imply
    \begin{equation}\label{eq:F_Fx_bounded}
        \varlimsup_{x\to0+}F(x) \le \frac{b_{i-1}}{a_{i-1}}
          < \frac{b_{i}}{a_{i}} \le \varliminf_{x\to+\infty}F(x)
        \an
        \varliminf_{x\to0+}\frac{F(x)}{x} \ge \frac{b_{i}}{a_{i-1}}
          > \frac{b_{i+1}}{a_i} \ge \varlimsup_{x\to+\infty}\frac{F(x)}{x},
    \end{equation}
    According to the latter inequality in~\eqref{eq:F_Fx_bounded},
    \[
    \frac{C_2}{C_1}\varliminf_{x\to0+}G(x)
    =\varliminf_{x\to0+}x^{j-k}F(x)
     >\varlimsup_{x\to+\infty}x^{j-k}F(x)=\frac{C_2}{C_1}\varlimsup_{x\to+\infty}G(x)
    \]
    \setstretch{1.2}%
    when~$j-k\le -1$ which contradicts to~\eqref{eq:G_bounded}. When~$j-k\ge 1$, the former
    inequality in~\eqref{eq:F_Fx_bounded} says that
    \[
    \frac{C_2}{C_1}\varlimsup_{x\to0+}\frac{G(x)}{x}=
    \varlimsup_{x\to0+}x^{j-k-1}F(x)< \varliminf_{x\to+\infty}x^{j-k-1}F(x)
    =\frac{C_2}{C_1}\varliminf_{x\to+\infty}\frac{G(x)}{x},
    \]
    which is inconsistent (due to~$\alpha_{-1}>\beta_1^{-1}$) with the estimate
    \[
    \alpha_{-1}-\frac{\alpha_{-1}^2x}{1+\alpha_{-1}x}
    < \frac{G(x)}{x}
    <\frac{1}{\beta_1} + \frac 1x,
    \ww x>0,
    \]
    equivalent to~\eqref{eq:G_bounded}. As a result, the only possible case is~$j=k$, that
    is~$F(z)$ coincides with~$G(z)$ up to a positive constant and can be expressed as
    in~\eqref{eq:order_alphas_betas}.
\end{proof}
\begin{lemma}\label{lemma:converse}
    If functions~$p(z)$ and~$q(z)$ have the form~\eqref{eq:funct_gen_dtps} and their
    ratio~$F(z)=\frac{q(z)}{p(z)}$ can be represented as in~\eqref{eq:order_alphas_betas}
    or~\eqref{eq:order_alphas_betas_mer}, then the matrix~$H(p,q)$ is totally nonnegative.
\end{lemma}
\begin{proof}
    Indeed, denote by~$p_*(z)\coloneqq \frac{p(z)}{g(z)}$
    and~$q_*(z)\coloneqq \frac{q(z)}{g(z)}$ the denominator and numerator of the function~$F(z)$
    given in~\eqref{eq:order_alphas_betas}. This means that~$p_*(z)\not\equiv 0$ and~$q_*(z)$
    have no common zeros, no poles and no exponential factors; the function~$g(z)$ has the
    form~\eqref{eq:funct_gen_dtps}. The function~$C\frac{q_n(z)}{p_n(z)}=F_n(z)$ introduced
    in~\eqref{eq:F_is_limit} maps the upper half-plane into itself for each positive
    integer~$n$. According to Theorem~3.44 of~\cite{HoltzTyaglov} (see also Theorem~1.4
    of~\cite{Dyachenko14} where the notation is closer to the current paper) the
    matrix~$H(p_n,q_n)$ is totally nonnegative. Since~$p_n(z)$ and~$q_n(z)$ converge
    in~$\mathbb C\setminus\{0\}$ locally uniformly to~$p_*(z)$ and~$q_*(z)$ respectively, their
    Laurent coefficients converge as well. Therefore, the matrix~$H(p_*,q_*)$ is totally
    nonnegative as an entry-wise limit of totally nonnegative matrices. Then the Cauchy-Binet
    formula implies the total nonnegativity of the
    matrix~$H(p,q)=H(p_*\,g,q_*\,g)=H(p_*,q_*)\cdot T(g)$, because~$T(g)$ is totally nonnegative
    by Theorem~\ref{th:E-AESW}.
\end{proof}
\begin{proof}[Proof of Theorem~\ref{th:main1}]
    Lemma~\ref{lemma:prop_S1} shows that the
    implications~\eqref{item:m3}$\implies$\eqref{item:m1}
    and~\eqref{item:m1}$\implies$\eqref{item:m3} follow, respectively, from
    Lemma~\ref{lemma:H_TNN_F_has_form} and Lemma~\ref{lemma:converse}.
\end{proof}
\addsec{Acknowledgements}
The author is grateful to Olga Holtz for two questions which gave rise to the present
publication.

{\setstretch{1}

}

\begin{thebibliography}{AESW51}\itemsep0pt\parsep0pt\small

\bibitem[AESW51]{AESW}
    \href{http://dx.doi.org/10.1073/pnas.37.5.303}{%
        M. Aissen, A. Edrei, I. J. Schoenberg, A. Whitney, `On the
            generating functions of totally positive sequences,' \emph{Proc. Nat. Acad. Sci. USA}
        \textbf{37} (1951), 303--307.}

\bibitem[ChM49]{ChebMei}
    N.~G.~Chebotarev (aka \v Cebotar\"ev or Tschebotareff) and N. N. Me\u\i man,
    {\itshape The Routh--Hurwitz problem for polynomials and entire functions.}
    Trudy Mat. Inst. Steklova XXVI, Acad. Sci. USSR, Moscow--Leningrad, 1949
    (Russian: {\otherlanguage{russian}\href{http://mi.mathnet.ru/eng/tm1039}{%
            Н.~Г.~Чеботарёв, Н.~Н.~Мейман,
            {\itshape Проблема Рауса-Гурвица для полиномов и целых функций.}\, Труды Мат. Инст.
            Стеклова XXVI, Изд-во АН СССР, Москва--Ленинград, 1949}}).

\bibitem[Du2004]{Duren2004} \href{http://dx.doi.org/10.1017/CBO9780511546600}{%
        {P. Duren}, {\itshape Harmonic mappings in the plane},
        {Cambridge Tracts in Mathematics 156}, Cambridge University Press, Cambridge, 2004.}

\bibitem[Dy2014]{Dyachenko14}%
    \href{http://dx.doi.org/10.1007/s11785-013-0344-0}{%
        {A. Dyachenko}, `Total nonnegativity of infinite Hurwitz matrices of entire and
        meromorphic functions', \emph{Complex Anal. Oper. Theory} \textbf{8} (2014), No.~5,
        1097--1127.}
    
\bibitem[Dy2016]{Dyachenko16b}%
    \href{http://arxiv.org/abs/1608.04440}{%
        {A. Dyachenko}, `Hurwitz matrices of doubly infinite series', \emph{preprint},
        arXiv:1608.04440\,[math.CV] (2016).}

\bibitem[Edr53]{Edrei}%
    \href{http://dx.doi.org/10.2307/1990808}{%
        A. Edrei, `On the generating function of a doubly infinite, totally positive
        sequence', \emph{Trans. Amer. Math. Soc.} \textbf{74} (1953), No.~3, 367--383.}

\bibitem[HT2012]{HoltzTyaglov}
    \href{http://dx.doi.org/10.1137/090781127}{%
        O. Holtz, M. Tyaglov, `Structured matrices, continued fractions, and root localization
        of polynomials', \emph{SIAM Rev.} \textbf{54}(3) (2012), 421--509.}
    
\bibitem[Kar68]{Karlin}%
    S. Karlin, \textit{Total Positivity.} Vol.~I. Stanford University Press, Stanford,
    California, 1968.

\bibitem[KaKr68]{KreinKac} I. S. Kac, M. G. Kre\u\i n, `R-functions --- analytic functions
    mapping the upper halfplane into itself', Supplement~I to the Russian edition of
    F.~V.~Atkinson, \textit{Discrete and continuous boundary problems.} Mir, Moscow, 1968,
    629--647 (Russian:
    {\otherlanguage{russian} И.~С. Кац, М.~Г. Крейн, `R-Функции --- аналитические функции,
        отображающие верхнюю полуплоскость в себя'. Дополнение~I к русскому изданию
        Ф.~В.~Аткинсон, \textit{Дискретные и непрерывные граничные задачи.} Мир, М., 1968,
        629--647}).
    English translation in: \emph{Amer. Math. Soc. Transl. (2)} \textbf{103} (1974), 1--18.    

\bibitem[Lev64]{Levin} B.~Ja.~Levin, \emph{Distribution of zeros of entire functions}. AMS, Providence,
    R.I. 1964.

\bibitem[Pi2010]{Pinkus}%
    \href{http://dx.doi.org/10.1017/CBO9780511691713}{%
        A.~Pinkus, \emph{Totally positive matrices}. Cambridge University Press, Cambridge,
        2010.}
   
\bibitem[Wa2000]{Wa}%
    \href{http://dx.doi.org/10.1017/S0963548399004162}{%
    D. G. Wagner, {\itshape Zeros of reliability polynomials and f-vectors of matroids}. Combin.
    Probab. Comput. \textbf{9} (2000), 167--190.}

\bibitem[Whi52]{Whitney}%
    \href{http://dx.doi.org/10.1007/BF02786969}{%
        A. M. Whitney, `A reduction theorem for totally positive matrices',
        \emph{J. Analyse Math.} \textbf{2} (1952), 88--92.}

\end{thebibliography}
\end{document}